\documentclass[letterpaper,12pt]{amsart}
\usepackage{latexsym,array,delarray,amsthm,amssymb,epsfig}

\theoremstyle{plain}
\newtheorem{thm}{Theorem}[section]
\newtheorem{lemma}[thm]{Lemma}
\newtheorem{prop}[thm]{Proposition}
\newtheorem{cor}[thm]{Corollary}

\newtheorem*{thm1*}{Theorem \ref{thm:theorem1}}
\newtheorem*{thm2*}{Theorem \ref{thm:theorem2}}
\newtheorem*{lemma*}{Lemma}
\newtheorem*{prop*}{Proposition}
\newtheorem*{cor*}{Corollary}
\newtheorem*{conj*}{Conjecture}

\theoremstyle{definition}
\newtheorem{defn}[thm]{Definition}
\newtheorem{ex}[thm]{Example}

\newtheorem{ques}[thm]{Question}

\theoremstyle{remark}
\newtheorem*{rmk}{Remark}

\newcommand{\rr}{\mathbb{R}}

\newcommand{\calc}{\mathcal{C}}

\newcommand{\ind}{\mbox{$\perp \kern-5.5pt \perp$}}

\newcommand{\captionfonts}{\small}
\makeatletter  
\long\def\@makecaption#1#2{%
  \vskip\abovecaptionskip
  \sbox\@tempboxa{{\captionfonts #1: #2}}%
  \ifdim \wd\@tempboxa >\hsize
    {\captionfonts #1: #2\par}
  \else
    \hbox to\hsize{\hfil\box\@tempboxa\hfil}%
  \fi
  \vskip\belowcaptionskip}
\makeatother   

\usepackage[svgnames]{xcolor}
\usepackage{ulem}
\usepackage[all,cmtip]{xy}

\begin{document}

\title{Identifiability results for several classes of linear compartment models}

\author{Nicolette Meshkat}
\author{Seth Sullivant}
\author{Marisa Eisenberg}

              \email{ncmeshka@ncsu.edu}
              \email{smsulli2@ncsu.edu } 
							\email{marisae@umich.edu}

              \address{Department of Mathematics, Box 8205, North Carolina State University, Raleigh, NC, 27695-8205, USA }
							\address{Departments of Epidemiology and Mathematics, University of Michigan,
               Ann Arbor, Michigan 48109-2029 } 
           

\maketitle

\begin{abstract}
Identifiability concerns finding which unknown parameters of a model can be estimated 
from given input-output data.  If some subset of the parameters of a model cannot 
be determined given input-output data, then we say the model is unidentifiable.
In past work we identified a class of models, that we call identifiable cycle models, 
which are not identifiable but have the simplest possible identifiable functions 
(so-called monomial cycles).  Here we show how to modify identifiable cycle models
by adding inputs, adding outputs, or removing leaks, in such a way that
we obtain an identifiable model.  We also prove a constructive result on how to
combine identifiable models, each corresponding to strongly connected graphs, into a larger 
identifiable model.  We apply these theoretical results to 
several real-world biological models from physiology, cell biology, and ecology. 
  
\par
\noindent \textit{Keywords:} Identifiability, Linear compartment models, Identifiable functions
\end{abstract}


\section{Introduction}

Parameter identifiability analysis for ODE models 
addresses the question of which unknown parameters can be quantified 
from given input-output data.  This paper is concerned with \textit{structural identifiability} 
analysis, that is, whether the model parameters can be identified from perfect input-output data (noise-free and of any duration required).
Structural identifiability is a necessary condition for \textit{practical identifiability}
which is identifiability analysis in the presence of noisy
and imperfect data.  Thus, structural identifiability is an important step in the parameter estimation problem, since failure to recover parameters in the ideal case implies failure in the imperfect case as well.

In the context of structural identifiability, if the parameters of 
a model have a unique or finite number of values 
given input-output data, then the model and its parameters are said 
to be \textit{identifiable}.  However, if some subset of the parameters 
can take on an infinite number of values and yet yield the same input-output data, 
then the model and this subset of parameters are called \textit{unidentifiable}.  
It is generally undesirable to work directly with an unidentifiable model,
and it is natural to ask what should be done (mathematically) when one is encountered.
One solution is to find the functions of parameters that can be identified
from given input-output data and reparametrize the model in terms of these
identifiable functions.
Another approach is to directly modify the model by reducing the number of fitted
parameters, e.g. by making some parameters depend on other parameters, or fixing parameters to a given value.
Finally, the input-output assumptions can be modified by adding inputs to the model,
adding outputs, or both.

In this paper, we focus on linear compartment models.  
Linear compartment models have an extensive history of practical use in many biological applications, including pharmacokinetics, toxicology, cell biology, physiology, and ecology \cite{Berman1956, Berman1962, DiStefano, Mulholland1974, Wagner1981}, dating back to the Widmark and Tandberg's first use of a one-compartment model with intravenous bolus injection and constant infusion inputs in 1924 \cite{Wagner1981, WidmarkTandberg}. Indeed, linear compartment models in pharmacokinetics are now fairly ubiquitous, with most kinetic parameters for drugs (such as half-lives, residence times, etc.) based at least in part on linear compartment model theory \cite{Tozer1981, Wagner1981}. While often the practical applications of linear compartment models in pharmacokinetics require only one or two compartments (e.g. representing blood plasma and tissue) \cite{Tozer1981, Wagner1981}, there are numerous examples which use larger, more complex compartment models, with many compartments accounting for different organs and metabolites (e.g. \cite{Berman1962, Birge1969, DArgenio1988, DiStefano1988, Feng1991}). 

In a typical biological application, the mass or concentration of a substance (e.g. drug concentration in an organ) is represented by a compartment, and the transfer of material from one compartment to another is given by a constant rate parameter, called an \textit{exchange rate}.  The transfer of material from a compartment leaving the system is given by a constant rate parameter called the \textit{leak rate}, and any compartment containing such a leak is called a \textit{leak compartment}.  An input represents the input of material to a particular compartment of the system (e.g. IV drug input) and an output represents a measurement from a compartment (e.g. drug concentration in an organ), where such compartments are called \textit{input compartments} and \textit{output compartments}, respectively. The resulting ODE system of equations (see Equation \ref{eq:main}) is linear.  This linearity feature has a nice mathematical consequence in that the model can be represented by a directed graph.  Thus, we can analyze identifiability problems in terms of the combinatorial 
structure of that graph.  In previous work \cite{MeshkatSullivant} the first and second
authors analyzed linear compartment models from the standpoint of
finding the identifiable functions in an unidentifiable model and reparametrizing
the model in terms of those functions.  A sufficient condition
on a graph (inductively strongly connected) was developed which allows for an easy method to deduce
that a simple identifiable reparametrization exists.  In the present paper,
we explore the extent to which combinatorial tools can be developed for the
other two approaches to dealing with unidentifiable models namely:  removing
leaks from a linear compartment model, and adding inputs and outputs.  

In particular, we will use as a starting point the models which we analyzed
in \cite{MeshkatSullivant}.  These are models that satisfy the following 
assumptions:

\begin{enumerate}
\item  Every compartment has a leak.
\item  There is a single input and output, and both are in the same compartment.
\item  The model corresponds to a strongly connected graph. 
\item  All the \textit{monomial cycles} in the graph are identifiable functions.
\end{enumerate}

We call a model that satisfies these conditions an \textit{identifiable cycle model}.
Conditions (1) and (2) are quite restrictive from the applications standpoint,
and we are primarily interested in understanding how to get rid of these conditions.
On the other hand, condition (4) is a very natural and useful condition mathematically. 
In \cite{MeshkatSullivant} the first and second authors showed that this condition
is equivalent to the existence of an identifiable scaling reparametrization, and showed a simple condition, based on the model graph being inductively strongly connected, which is sufficient to ensure (4).  
Our main question in this work is the following:

\begin{ques}
Given an identifiable cycle model, how many leaks should be removed or how 
many inputs/outputs should be added in order to render it identifiable?
\end{ques}

Our motivation is to perform identifiability analysis on models 
that are used in applications.  Specific such applications are
described in Section \ref{sec:examples}.
Here is an example of the type of general result that we will derive 
that can be used to prove identifiability of various biological models.

\begin{thm} \label{thm:theorem1} Let $M$ be an identifiable cycle model.  
If the model is changed to have exactly one leak, 
then the resulting model is locally identifiable.
\end{thm} 
We will also provide in Section \ref{sec:generalize} an extension of this result which contains a combination of removing leaks and adding inputs or outputs.  In Section \ref{sec:extensions}, we prove that combining two strongly connected identifiable submodels into a larger model with a one-way flow results in an identifiable model.

Previous work on finding sufficient conditions to obtain identifiability has been done in both the linear and nonlinear setting.  Vicini et al \cite{Vicini} find conditions for mammillary and catenary linear compartmental models to be globally identifiable.  We note that mammillary and catenary models correspond to inductively strongly connected graphs and thus belong to a subset of the models handled in Theorem \ref{thm:theorem1}.  Anguelova et al \cite{Anguelova} find a minimal set of outputs to guarantee local identifiability of a nonlinear model.  While they provide a much more general result for nonlinear models, their result only applies to adding outputs and does not investigate the effect of adding inputs.  Our result in Section \ref{sec:generalize} thus allows one to obtain identifiability by adding a combination of inputs and outputs.  Vajda \cite{Vajda} gives necessary and sufficient conditions for a linear compartmental model and its submodels to be globally identifiable using structural equivalence.  Our result in Section \ref{sec:extensions} only requires that our submodels be strongly connected, whereas the result in \cite{Vajda} requires additional assumptions.

An important caveat we would like to address is that our main intention is to give identifiability results for families of linear compartment models.  The theorems and proofs are constructed in terms of starting with an identifiable cycle model and then adjusting it to evaluate other models with the same graphical `backbone' of exchanges. However, clearly an adjusted model cannot be applied for the same biological application as the original identifiable cycle model.  Thus, for biological purposes, our results can be applied in reverse order.  In other words, for given biological model, if the equivalent model with the same `backbone' but with leaks added to all compartments and only one combined input/output compartment yields an identifiable cycle model, this property can dictate the identifiability of our original given biological model.

The organization of the paper is as follows. The next section provides introductory material on compartment models and how to derive the input-output equations.  Section \ref{sec:ident} gives definitions of identifiability and identifiable cycle models.  Section \ref{sec:leaks} explains how, starting with an identifiable cycle model, removing all leaks except one results in identifiability.  Section \ref{sec:generalize} explains how, starting with an identifiable cycle model, removing a subset of leaks and adding
inputs/outputs results in identifiability.  Section \ref{sec:extensions} demonstrates how to combine identifiable submodels into a larger model and obtain identifiability of the full model. Section \ref{sec:examples} includes various examples of biological models satisfying our properties.


\section{Linear Compartment Models and Their Input-Output Equations}\label{sec:model}

In this section, we introduce general linear compartment models and 
explain how to calculate their input-output equations.
In the next section, we will describe the identifiability problem
for these models.  

Let $G$ be a directed graph with vertex set $V$ and set of 
directed edges $E$.  Each vertex $i \in V$ corresponds to a
compartment in our model and an edge $j \rightarrow i$ denotes 
a direct flow of material from compartment $j$ to
compartment $i$.  Also introduce three subsets of the vertices
$In, Out, Leak \subseteq V$ corresponding to the
set of input compartments, output compartments, and leak compartments
respectively.  To each edge $j \rightarrow i$ we associate
an independent parameter $a_{ij}$, the rate of flow
from compartment $j$ to compartment $i$.  
To each leak node $i \in Leak$, we associate an independent
parameter $a_{0i}$, the rate of flow from compartment $i$ leaving the system.

We associate a matrix $A(G)$ to the graph and the set $Leak$
 in the following way:
\[
  A(G)_{ij} = \left\{ 
  \begin{array}{l l l}
    -a_{0i}-\sum_{k: i \rightarrow k \in E}{a_{ki}} & \quad \text{if $i=j$ and } i \in Leak\\
        -\sum_{k: i \rightarrow k \in E}{a_{ki}} & \quad \text{if $i=j$ and } i \notin Leak\\
    a_{ij} & \quad \text{if $j\rightarrow{i}$ is an edge of $G$}\\
    0 & \quad \text{otherwise}\\
  \end{array} \right.
\]
For brevity, we will often
 use $A$ to denote $A(G)$.
Then we construct a system of linear ODEs with inputs and outputs associated to the quadruple
$(G, In, Out, Leak)$ as follows:
\begin{equation} \label{eq:main}
\dot{x}(t)=Ax(t)+u(t)  \quad \quad y_i(t)=x_i(t)  \mbox{ for } i \in Out
\end{equation}
 where $u_{i}(t) \equiv 0$ for $i \notin In$.
 The coordinate functions $x_{i}(t)$ are the state variables, the 
 functions $y_{i}(t)$ are the output variables, and the nonzero functions $u_{i}(t)$ are
 the inputs.  The resulting model is called a   \textit{linear compartment model}.  

We use the following convention for drawing linear compartment models \cite{DiStefano}.  Numbered vertices represent compartments, outgoing arrows
from the compartments represent leaks, an edge with a circle coming out of a compartment represents an output, 
and an arrowhead pointing into a compartment represents an input.
 
\begin{figure}
\begin{center}
\resizebox{!}{3cm}{
\includegraphics{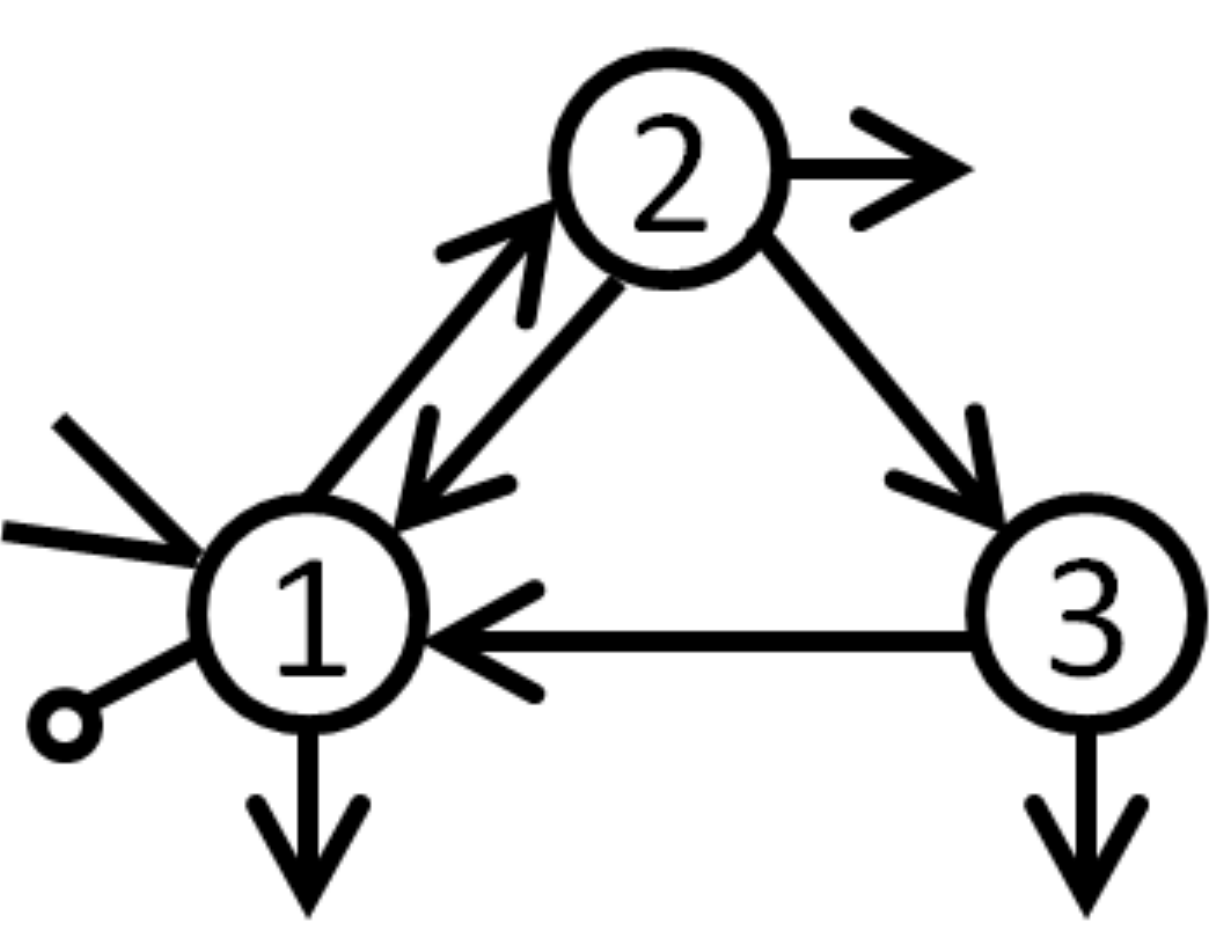}}
\end{center}\caption{A 3-compartment model}
\label{fig:3comp3leaks}
\end{figure}

\begin{ex}
For the compartment model in Figure \ref{fig:3comp3leaks}, the ODE system has the following form:

$$
\begin{pmatrix} 
\dot{x}_1 \\
\dot{x}_2 \\
\dot{x}_3 \end{pmatrix} = {\begin{pmatrix} 
-(a_{01}+a_{21}) & a_{12} & a_{13} \\
a_{21} & -(a_{02}+a_{12}+a_{32}) & 0 \\
0 & a_{32} & -(a_{03}+a_{13}) 
\end{pmatrix}} {\begin{pmatrix}
x_1 \\
x_2 \\
x_3 \end{pmatrix} } + {\begin{pmatrix}
u_1 \\
0 \\
0 \end{pmatrix}}
$$
$$ y_1=x_1.$$
\end{ex}
 
We now define the concepts of $\textit{strongly connected}$ and $\textit{inductively strongly connected}$, which will be important concepts in this paper. 

\begin{defn} A directed graph $G$ is \textit{strongly connected} if there exists a directed path from each vertex to every other vertex.  A directed graph $G$ is \textit{inductively strongly connected} with respect to vertex $1$ if each of the induced subgraphs $G_{\{1, \ldots, i\}}$ is strongly connected for $i = 1, \ldots, |V|$ for some ordering of the vertices $1,\ldots,i$ which must start at vertex $1$.
\end{defn}

\begin{ex} The graph in Figure \ref{fig:3comp3leaks} is inductively strongly connected for the ordering of vertices $\{1,2,3\}$, so that $G_{\{1,2\}}$ is strongly connected, and $G_{\{1, 2, 3\}}$ is strongly connected.
\end{ex}

We assume that we can only observe the inputs $u_j$ and the outputs $y_i$
with $j \in In$ and $i \in Out$.
The state variable $x$ and the parameter entries of $A$ are unobserved and unknown.  
Since we can only observe the input and output to the system, 
we are interested in relating these quantities by forming the 
\textit{input-output equations}, i.e.~equations purely in terms of 
input, output, and parameters that describe the dynamics of the inputs
and outputs alone.  The structure of the input-output equations
plays a significant role in the structural identifiability problem,
as we explain in the next section. 

There have been several methods proposed to find the input-output 
equations of nonlinear ODE models 
\cite{Audoly2001, EvansChappell, Ljung, Meshkat2, Pohjanpalo}, 
but for linear models the problem is much simpler.  We use Cramer's rule:

\begin{thm}\label{thm:ioeqn}
Let $\partial$ be the differential operator $d/dt$ and let 
$A_{ji}$ be the submatrix of $\partial{I}-A$ obtained by deleting the $jth$ row and the $ith$ column
of $\partial{I}-A$.  Then the input-output equations are of the form:
$$\frac{\det(\partial{I}-A)}{g_i}y_i=\sum_{j \in In} (-1)^{i + j} \frac{\det(A_{ji})}{g_i} u_j$$
where $g_i$ is the greatest common divisor of  $\det(\partial{I}-A)$, $\det(A_{ji})$ such that $j \in In$ for a given $i \in Out$.
\end{thm}

\begin{proof} 
We can formally rewrite our ODE system as:
$$
(\partial{I}-A)x=u.
$$
Using formal manipulations with this operator, we can use Cramer's Rule to get that 
$$
x_i = \det(B_i)/\det(\partial{I}-A)
$$
where $B_i$ is the matrix $(\partial{I}-A)$ with the $ith$ column replaced by $u$. 
Since $u$ has nonzero entries $u_j$ when $j \in In$ then $\det(B_i)$ 
can be expanded using the Laplace expansion down the $i$th column 
(over $j \in In$) of $(-1)^{i +j}\det(A_{ji})u_j$.   
Then replacing $x_i$ with $y_i$, we get the input-output equation:
$$
\det(\partial{I}-A)y_i=\sum_{j \in In} (-1)^{i +j}\det(A_{ji})u_j.
$$ 
This equation is not necessarily minimal, and we must remove common factors if they appear.  This is
the reason for the division by the greatest common divisor $g_i$ in the statement
of the theorem.
\end{proof}  

We now find a sufficient condition for the greatest common divisor $g_i$ in Theorem \ref{thm:ioeqn} to be $1$.

\begin{prop} \label{prop:detnonzero} Let $G$ be strongly connected.  If there is at least one leak, then $\det(A)$ is nonzero.
\end{prop}

\begin{proof} If there is at least one leak, then the matrix $A$ is weakly diagonally dominant.  Since the graph $G$ is strongly connected, then the matrix $A$ is irreducible, meaning it is not similar via a permutation to a block upper triangular matrix.  Thus, the matrix $A$ is irreducibly diagonally dominant, which means it is full rank, and thus $\det(A)$ is nonzero.
\end{proof}

\begin{thm} \label{thm:irred} Let $G$ be strongly connected.  If there is at least one leak, then for a generic choice of parameters, the characteristic polynomial $\det(\lambda{I}-A)$ is irreducible.
\end{thm}

\begin{proof} Without loss of generality, assume there is just one leak from the first compartment.  Let $|V|=n$ and let $B=\lambda{I}-A$.  Then $B$ can be written as $\lambda{I}-L(G)+A'$, where $L(G)$ is the negative of the Laplacian matrix of $G$ and $A'$ is the matrix with the $(1,1)$ entry equal to $-a_{01}$ and all other entries equal to zero.  Let the submatrix corresponding to the second through $n^{th}$ columns of $\lambda{I}-L(G)$ be denoted $L'$.  Then $\det(B)$ can be written as $\det(\lambda{I}-L(G))+\det(A'')$ where $A''$ has its first column equal to the first column of $A'$ and the second through $n^{th}$ columns equal to $L'$.  Let $char(\lambda,A)$ denote the characteristic polynomial of the matrix $A$.  Then $\det(B)$ can be simplified as $char(\lambda,L(G))-a_{01}char(\lambda,L_{11}(G))$, where $L_{11}(G)$ is the submatrix of $L(G)$ formed by removing the first row and first column of $L(G)$.  If $\det(B)$ were reducible, then it can be factored as $(f_1(A,\lambda)+a_{01}f_2(A,\lambda))g(A,\lambda)$ for some polynomials $f_1,f_2,g$ where $g(A,\lambda)$ divides both $char(\lambda,L(G))$ and $char(\lambda,L_{11}(G))$.  Note that $\lambda = 0$ is a root of $char(\lambda,L(G))$, so that the constant term of $char(\lambda,L(G))$ is zero.  We now show that the coefficient of $\lambda$ in $char(\lambda,L(G))$ is an irreducible polynomial, so that $char(\lambda,L(G))/\lambda$ is irreducible.  Since $\det(A)$ is nonzero by Proposition \ref{prop:detnonzero}, then $g(A,\lambda)$ cannot be $\lambda$, and thus $char(\lambda,L(G))/\lambda$ being irreducible implies that $g(A,\lambda)$ must be a constant.

The coefficient of $\lambda$ in $char(\lambda,L(G))$ can be expressed as the sum over all rooted intrees in the directed graph $G$ \cite[Lemma ~5.6.5]{Stanley}.  Since $G$ is strongly connected, then there is an intree with a rooting at each node in the graph.  Define the degree of $a_{ji}$ to be $e_i$, where $e_i$ denotes the $i^{th}$ standard unit vector.  Then we define the degree of a monomial to be the sum of the degrees of all of the parameters $a_{ji}$ in the monomial.  This means the degree for an intree is $1-e_i$, where $1$ is the ones vector and $i$ is the rooting vertex.  Thus, the set of all degrees of the monomials that occur in the coefficient of $\lambda$ in $char(\lambda,L(G))$ is $\{1-e_i : i=1,...,n\}$.  If this polynomial factored, then $\{1-e_i : i=1,...,n\}$ could be written as $A+B$ where $A+B=\{a+b : a \in A, b \in B\}$ for distinct degree sets $A$ and $B$.  However there is no such $A$ and $B$ since there is an intree with a rooting at each node, thus there exists no entry $j$ of $1-e_i$ that is equal to $1$ for every $i$.  Thus, the coefficient of $\lambda$ in $char(\lambda,L(G))$ is an irreducible polynomial, so that $char(\lambda,L(G))/\lambda$ is irreducible.

\end{proof}

Thus, Theorem \ref{thm:irred} gives us the following corollary on when the greatest common divisor $g_i$ from Theorem \ref{thm:ioeqn} is just one:

\begin{cor} \label{cor:relprime} Let $G$ be strongly connected.  Let $\partial$ be the differential operator $d/dt$ and let 
$A_{ji}$ be the submatrix of $\partial{I}-A$ obtained by deleting the $jth$ row and the $ith$ column
of $\partial{I}-A$.  If there is at least one leak, then the input-output equations are of the form:
$${\det(\partial{I}-A)}y_i=\sum_{j \in In} (-1)^{i + j} {\det(A_{ji})} u_j$$
\end{cor}

\begin{ex}
The model in Figure \ref{fig:3comp3leaks} has the following input-output equation, where we let $a_{11}=-(a_{01}+a_{21}),a_{22}=-(a_{02}+a_{12}+a_{32}),a_{33}=-(a_{03}+a_{13})$:

\begin{align*}
y_1^{(3)}-E_{1}(a_{11},a_{22},a_{33})y_1^{(2)}  
+ ( E_{2}(a_{11},a_{22},a_{33})-a_{12}a_{21})y_1^{'} \\
-(E_{3}(a_{11},a_{22},a_{33})
-a_{12}a_{21}a_{33} +a_{21}a_{13}a_{32})y_1 \\ =  \quad \quad u_{1}^{(2)}-E_{1}(a_{22},a_{33})u_{1}^{'}
+E_{2}(a_{22},a_{33})u_{1}
\end{align*}
where $E_{k}(z_{1}, \ldots, z_{m})$ denotes the $k$-th elementary symmetric polynomial
in $z_{1}, \ldots, z_{m}$.
\end{ex}


\section{Identifiability of Input-Output Equations}\label{sec:ident}

A state space model with given inputs and outputs is said
to be \textit{generically structurally identifiable} if with a generic choice of the inputs
and initial conditions, the parameters of the model can be recovered 
from exact measurements of both the inputs and the outputs.
In this section we explain the precise mathematical content
of this notion.

\begin{defn}\label{defn:identify}
Let $(G, In, Out, Leak)$ be a linear compartment model and
let $c$ denote the vector of all coefficient functions of
all the linear input-output equations derived in Theorem
\ref{thm:ioeqn} for each $i \in Out$.  The function $c$ defines a map
$c:  \rr^{|E| + |Leak|}  \rightarrow \rr^{k}$,
where $k$ is the total number of coefficients.  The linear compartment model  $(G, In, Out, Leak)$ is: 
\begin{itemize}
	\item{\textit{globally identifiable} if $c$ is a one-to-one function, and is \textit{generically globally identifiable} if global identifiability holds everywhere in $\rr^{|E| + |Leak|}$, except possibly on a set of measure zero.}
	\item{\textit{locally identifiable} if around any neighborhood of a point in $\rr^{|E| + |Leak|}$, $c$ is a one-to-one function, and is \textit{generically locally identifiable} if local identifiability holds everywhere in $\rr^{|E| + |Leak|}$, except possibly on a set of measure zero.}
\item{\textit{unidentifiable} if $c$ is infinite-to-one.}
\end{itemize}
\end{defn}

\begin{ex}
The model in Figure \ref{fig:3comp3leaks} has $|E|+|Leak|=4+3=7$ independent parameters, but only $5$ coefficients, so the map $c$ from $7$ parameters to $5$ coefficients is infinite-to-one.  Thus the model is unidentifiable.
\end{ex}

It is a general fact from differential algebra that
it is possible to recover all the coefficients of the input-output
equations from the arbitrary data with perfect measurements \cite{Ljung}.
Hence, Definition \ref{defn:identify} gives the practical
essence of the definition of identifiability.  In other words, identifiability concerns whether it is possible to recover the parameters of a model given the dynamics from the input-output equations.  In the standard differential algebra approach of \cite{Saccomani2003}, one finds a \textit{characteristic set} of the system in Equation \ref{eq:main}, which contains possibly lower order input-output equations containing more than one output.  We note that finding an input-output equation in $y_i$ for each $i \in Out$ using Theorem \ref{thm:ioeqn} also gives a full description of the dynamics of the model, thus our identifiability results are consistent with the approach in \cite{Saccomani2003}. 
 

In this paper we focus almost exclusively on ``generic'' local identifiability and will use the following result to determine generic local identifiability.  

\begin{prop} \label{prop:jacobian}
The model $(G, In, Out, Leak)$ is generically locally
identifiable if and only if the rank of the Jacobian of $c$
is equal to $|E| + |Leak|$ when evaluated at a random point. 
\end{prop}

\begin{proof}
Since the coefficients in $c$ are all polynomial functions of the
parameters, the model $(G, In, Out, Leak)$ is generically locally
identifiable if and only if the image of $c$ has dimension equal to
the number of parameters, i.e.~$|E| + |Leak|$.  The dimension of the image of
a map is equal to the evaluation of the Jacobian at a generic point.
\end{proof}

Recent work from Baaijens \cite{Baaijens} shows that the criterion in Proposition \ref{prop:jacobian} can be translated to one based on the rank of the bi-adjacency matrix which is order $O(|V|^6)$ operations, compared to the method in Proposition \ref{prop:jacobian} which is $O(|V|^8)$ operations \cite{Baaijens}.

In previous work \cite{MeshkatSullivant}, the first and second authors
focused on a special case of linear compartment models, which we call here \textit{identifiable cycle}
models.

\begin{defn}
A linear compartment model $(G, In, Out, Leak)$
is called an \textit{identifiable cycle model} if
\begin{enumerate}
\item  $G$ is strongly connected,
\item  $In = Out  = \{1\}$,
\item  $Leak = V$,
\item  the dimension of the image of $c$ is $|E| + 1$.
\end{enumerate}
\end{defn}

\begin{rmk} For a linear compartment model satisfying conditions (1), (2), and (3) where $|E|=2|V|-2$, a sufficient condition to obtain condition (4) is that the graph $G$ be inductively strongly connected with respect to vertex $1$ \cite{MeshkatSullivant}.  Thus, whether a model is an identifiable cycle model can be checked without calculating $\dim {\rm image} \, c$ for inductively strongly connected graphs.  However, this condition is sufficient but not necessary, thus a graph $G$ that fails to be inductively strongly connected does not imply condition (4) fails as well.  We note that checking if a graph is inductively strongly connected has order $O(|V|^3)$ operations, but has the additional benefit that it can be done by visual inspection, as opposed to other methods (e.g. Laplace Transform \cite{Bellman}, Differential algebra method \cite{Ljung}) for checking identifiability.
\end{rmk}

As shown in \cite{MeshkatSullivant}, the largest dimension of the image of
the coefficient map $c$ for a linear compartment model 
satisfying conditions (1), (2), and (3) of the definition is $|E| +1$
which is less that $|E| + |Leak|$, so such a model is never identifiable.
However, identifiable cycle models have identifiable 
scaling reparametrizations in terms
of the monomial cycles that are identifiable in the model.  To
explain this in detail we need the definition of an identifiable function.

\begin{defn} \label{defn:idfunction}
Let $c$ be a function $c:\Theta\rightarrow{\rr^{m_{2}}}$, where
$\Theta \subseteq \rr^{m_{1}}$.  A function
$f : \Theta \rightarrow \rr$ is \textit{globally identifiable} from
$c$ if there exists a function $\Phi: \rr^{m_{2}} \rightarrow \rr$
such that $\Phi \circ c  = f$.  The function $f$ is
\textit{locally identifiable} if there is a finitely multivalued function
$\Phi: \rr^{m_{2}} \rightarrow \rr$
such that $\Phi \circ c  = f$.
\end{defn}

A function $f$ being \textit{generically} locally identifiable from a map $c$ can be phrased in terms of dimension calculations using the Jacobian of $c$.

\begin{prop}
Let $c:\Theta\rightarrow{\rr^{m_{2}}}$, where
$\Theta \subseteq \rr^{m_{1}}$ is a $m_{1}$ dimensional subset of $\rr^{m_{1}}$.
A function $f: \Theta \rightarrow \rr$ is generically locally identifiable from $c$ if
the vector $\nabla f$ is in the span of the rows of $J(c)$.
Equivalently, consider the map $(c,f):  \Theta \rightarrow \rr^{m_{2}+1}$.
Then $f$ is generically locally identifiable from $c$ if and only if
the dimension of the image of $(c,f)$ equals the dimension of the image
of $c$.
\end{prop}

\begin{proof}
This follows from Definition \ref{defn:idfunction} and the fact that the dimension of the image of
a map is equal to the evaluation of the Jacobian at a generic point.
\end{proof}

For a model $(G, In, Out, Leak)$ where there is a leak in every
compartment (i.e. $Leak = V$), it can greatly simplify the representation
to use the fact that the diagonal entries of $A(G)$ are the
only places where the parameters $a_{0i}$ appear.
Since these are algebraically independent parameters,
we can introduce a new algebraically  independent 
parameter $a_{ii}$ for the diagonal entries (i.e.
we make the substitution $a_{ii}  =  -a_{0i}-\sum_{k: i \rightarrow k \in E}{a_{ki}}$)
to get generic parameter values along the diagonal.
Identifiability questions in such a model are equivalent
to identifiability questions in the model with this reparametrized
matrix.  In the case of an identifiable cycle model, the identifiable functions
are explicitly characterized in terms of the combinatorics of $G$.

\begin{defn} A \textit{closed path} in a directed graph $G$ is a sequence of 
vertices $i_{0},i_{1}, i_{2}, \ldots, i_{k}$ with $i_{k} = i_{0}$ and
such that $i_{j+1} \to i_{j}$ is an edge for all $j = 0, \ldots, k-1$.
A \textit{cycle} in $G$ is a closed path with no repeated vertices.
To a cycle $C = i_{0},i_{1}, i_{2}, \ldots, i_{k}$, we associate the
monomial $a^{C} = a_{i_{1}i_{2}}a_{i_{2}i_{3}}\cdots a_{i_{k}i_{1}}$,
which we refer to as a \textit{monomial cycle}.  If a monomial cycle $a^{C}$ has
length $k$, we refer to it as a $k$-cycle.  
\end{defn}

Note that we also include the monomial cycles $a_{ii}$ which
are $1$-cycles, or \textit{self-cycles}.  We now state our main result from \cite{MeshkatSullivant}:

\begin{thm}\label{thm:oldstandardmodels}
Let $(G,\{1\}, \{1\}, V)$ be an identifiable cycle model.  Then every monomial
cycle in $G$ is identifiable and every identifiable function is
a function of the monomial cycles.  Equivalently, there exists an identifiable scaling reparametrization in terms of monomial functions of the original parameters.
\end{thm}

\begin{ex}
The model in Figure \ref{fig:3comp3leaks} is an identifiable cycle model, thus the monomial cycles $a_{11},a_{22},a_{33},a_{12}a_{21},a_{21}a_{13}a_{32}$ are all identifiable.  Using the scaling $X_1=x_1$, $X_2=x_{2}/a_{21}$, and $X_3=x_{3}/a_{21}a_{32}$, we get the following identifiable scaling reparametrization:

$$
\begin{pmatrix} 
\dot{X}_1 \\
\dot{X}_2 \\
\dot{X}_3 \end{pmatrix} = {\begin{pmatrix} 
a_{11} & a_{12}a_{21} & a_{21}a_{13}a_{32} \\
1 & a_{22} & 0 \\
0 & 1 & a_{33} 
\end{pmatrix}} {\begin{pmatrix}
X_1 \\
X_2 \\
X_3 \end{pmatrix} } + {\begin{pmatrix}
u_1 \\
0 \\
0 \end{pmatrix}}
$$
$$ y_1=X_1.$$
\end{ex}

In \cite{MeshkatSullivant}, the goal was to determine identifiable scaling reparametrizations over the identifiable monomial cycles in the graph $G$.  In this work, our main goal is to try to determine sufficient conditions on the position of inputs, outputs, and leaks to obtain identifiability, starting with an identifiable cycle model.  We will need the following definitions and results from \cite{MeshkatSullivant} in order to do this.

We define the $|V|$ by $|E|$ \textit{incidence matrix} $M(G)$ as:

\begin{equation}\label{eq:eg}
  M(G)_{i,(j,k)} = \left\{ 
  \begin{array}{l l l}
    1 & \quad \text{if $i=j$}\\
   -1 & \quad \text{if $i=k$}\\
    0 & \quad \text{otherwise.}\\
  \end{array} \right.
\end{equation}
In other words, $M(G)$ has column vectors corresponding to the edges 
$j\rightarrow{k} \in E$ with a $1$ in the $jth$ row, $-1$ in the $kth$ row, and $0$ 
otherwise.  We define the \textit{indicator vector} of a cycle $C$ as the vector $(x_s)_{s \in E}$ such that $x_s = 1$ if $s \in E_C$ and $x_s = 0$ if $s \notin E_C$, where $E_C$ is the set of edges associated to the cycle $C$.  We state one final result from \cite{MeshkatSullivant} to be used in later sections:

\begin{prop} \label{prop:kernelincidence} Let $G$ be a strongly connected graph.  Then a set of $|E|-|V|+1$ linearly independent indicator vectors of cycles form a basis for the kernel of $M(G)$.  
\end{prop}

It follows that there are $|E|-|V|+1+|V|=|E|+1$ independent monomial cycles and self-cycles corresponding to a model $(G, In, Out, V)$ where $G$ is strongly connected \cite{MeshkatSullivant}.


\section{Obtaining identifiability by removing leaks}\label{sec:leaks}

The main idea in this section is that, starting with an identifiable cycle model, 
removing exactly $|V|-1$ leaks from all compartments except one gives local identifiability of the model.  
More precisely, removing $|V|-1$ leaks has the effect of decreasing the number 
of parameters from $|V|+|E|$ down to $|V|+|E|-(|V|-1)=|E|+1$, while preserving the 
dimension of the image of $c$ as $|E| + 1 = |E| + |Leak|$.

We now re-state our main theorem, Theorem \ref{thm:theorem1}, 
and spend the rest of the section proving it.

\begin{thm1*} Let $(G, \{1\}, \{1\}, V)$ represent an identifiable cycle model. 
Then the corresponding model with a leak in a single compartment
 $(G, \{1\}, \{1\}, \{k\})$ is  generically locally identifiable.
\end{thm1*}

To prove Theorem \ref{thm:theorem1}, we will make use of Theorem \ref{thm:oldstandardmodels},
and the particular structure of the coefficient map $c$ that arises when
considering an identifiable cycle model.
In particular, for an identifiable cycle model we have the coefficient map
$c: \rr^{|V| + |E|}  \rightarrow \rr^{2|V| -1}$.
Let $f :  \rr^{|V| + |E|} \rightarrow \rr^{|E| + 1}$ be the \textit{cycle map}
that is, $f(A)  =  (a^{C} : C \mbox{ is a cycle of } G )$.
Then Theorem \ref{thm:oldstandardmodels} tells us that $c$ factors through
$f$, without a loss of dimension.  That is, there exists
a map $\Psi :  \rr^{|E| + 1} \rightarrow \rr^{2|V| -1}$
such that we have the following commutative diagram
$$\xymatrix{
\rr^{|V| + |E|}  \ar[r]^-{\displaystyle c} \ar@/_/[rd]^{\displaystyle f} & 
**[r] \rr^{2|V|-1}  \\
& **[r] \rr^{|E| + 1} \ar[u]_{\displaystyle \Psi}
}
$$
with the property that $c = \Psi \circ f$ and $\dim {\rm image} \, c = \dim {\rm image} \, f$.

Passing from a model $(G, \{1\}, \{1\}, V)$ to a model $(G, \{1\}, \{1\}, \{k\} )$
amounts to restricting the parameter space $\rr^{|V| + |E|}$
to a linear subspace $\Omega \subseteq  \rr^{|V| + |E|}$ of dimension $|E| + 1$
and we would like the image of $\Omega$ under the coefficient map $c$
to have dimension $|E| + 1$.  Since  $c$ factors through $f$ it suffices to prove
that the image of $\Omega$ under $f$ has dimension $|E| + 1$.

\begin{lemma}\label{lem:sliceparam}
Let $G = (V,E)$ be a directed graph with corresponding identifiable cycle model $(G, \{1\}, \{1\}, V )$.  Consider a model $(G, \{1\}, \{1\}, \{k\})$.  Let $\calc$ be a set of
cycles in $G$ that are linearly independent and span the cycle space.
Let
$f :  \rr^{|V| + |E|}  \rightarrow \rr^{|E| + 1}  :  A  \mapsto  (a^C)_{C \in \mathcal{C}}$
be the cycle map.  Let $\Omega \subseteq \rr^{|V| + |E|}$ be the linear space satisfying
$$
\Omega = \{ A \in \rr^{|V| + |E|}: a_{ii} =  - \sum_{j,j \neq i} a_{ji} \mbox{ for all }
i \neq k \}.
$$
Then the dimension of the image of $\Omega$ under the map $f$ is $|E| + 1$.
\end{lemma}

\begin{proof}
First of all, since $\Omega$ is a linear space, we can replace $f$ by the natural map  
from $\rr^{|E| + 1}  \rightarrow \rr^{|E| + 1}$.  We also call this map $f$.
To show that the dimension is correct, we compute the Jacobian of $f$
and show that it has full rank.

Clearly, the row corresponding to the cycle $a_{kk}$ is linearly independent 
of the rest of the rows, so we focus on the $|E|$ by $|E|$ submatrix 
ignoring that row and column, and call this matrix $J$.  
Arrange the matrix so that the first $|E|-|V|+1$ rows correspond to the 
(non-diagonal) cycles of $G$ and the last $|V|-1$ rows correspond to the 
non-leak diagonal elements.  Let the first $|E|-|V|+1$ rows be called 
$A$ and the last $|V|-1$ rows be called $B$.  Clearly the rows of 
$A$ are linearly independent (by assumption).  The rows of $B$ are 
linearly independent since they are in triangular form (each involves 
distinct parameters).  We want to show that the full set of 
$|E|$ rows are linearly independent.  To do this, we show that 
the row space of $A$ and the row
space of $B$ intersect only in the origin.

To prove that $J$ generically has the maximal possible rank, it suffices to find
some point where the evaluation of $J$ at this point yields
the maximal possible rank.  We choose the point where we
set all the edge parameters $a_{ij} = 1$ for all $ j \to i \in E$.
This has the following effect on the problem
of comparing the row space of $A$ and the row space of $B$:
the row space of $A$ is exactly the cycle space of the graph $G$,
that is it consists of all weightings on the edges of the graph
where the indegree equals the outdegree of every vertex. On the other hand,  the matrix
$B$ is a $(|V|-1) \times |E|$ matrix.  The rows correspond to the vertices 
in $V \setminus \{k \}$, and the (negated) row corresponding to vertex $i$
has a one for an edge $i' \to j'$ if and only if $i = i'$, with all other entries zero.  

Since $A$ spans the cycle space of $G$, each element in the row space of
$A$ corresponds to a weighting on the edges of $G$ where the total
weight of all incoming edges at a vertex $i$ equals the total weight of all
outgoing edges at vertex $i$.  On the other hand, the only vector
in the row span of $B$ with the same property is the zero vector.
To see this, let $b_{i}$ be the row vector of $B$ associated to vertex
$i$.  Now a vector in the row span of $B$ will have zero weight 
on any of the outgoing edges of vertex $k$.  Hence, we could not include
a nonzero multiple of any $b_{i}$ with an edge pointing into vertex $k$.
By similar reasoning, this precludes the inclusion of any $b_{j}$ such
that $j \to i \to k$ in the graph.  By induction, and the fact that
$G$ is strongly connected, we must have weight zero on each of the
vectors $b_{j}$.
\end{proof}

\begin{proof}[Proof of Theorem \ref{thm:theorem1}]
By Lemma \ref{lem:sliceparam} and the comments preceding it
we know that the image of the restricted parameter space
under the cycle map $f$ has dimension $|E| + 1$, which 
is equal to the dimension of the image of the full parameter
space under the cycle map.  Since, for an identifiable cycle model, the
dimension of the image of the coefficient map $c$ is $|E| + 1$,
this must be the same for the restricted model.  But the model
has $|E| + 1$ parameters, hence it is generically locally identifiable.
\end{proof} 

We now demonstrate Theorem \ref{thm:theorem1} on a three compartment model with input, output, and leak all from the first compartment.

\begin{ex} \label{ex:mainex} 
Our ODE system has the following form:

$$
\begin{pmatrix} 
\dot{x}_1 \\
\dot{x}_2 \\
\dot{x}_3 \end{pmatrix} = {\begin{pmatrix} 
-(a_{01}+a_{21}) & a_{12} & a_{13} \\
a_{21} & -(a_{12}+a_{32}) & 0 \\
0 & a_{32} & -a_{13} 
\end{pmatrix}} {\begin{pmatrix}
x_1 \\
x_2 \\
x_3 \end{pmatrix} } + {\begin{pmatrix}
u_1 \\
0 \\
0 \end{pmatrix}}
$$
$$ y_1=x_1.$$

\begin{figure}
\begin{center}
\resizebox{!}{3cm}{
\includegraphics{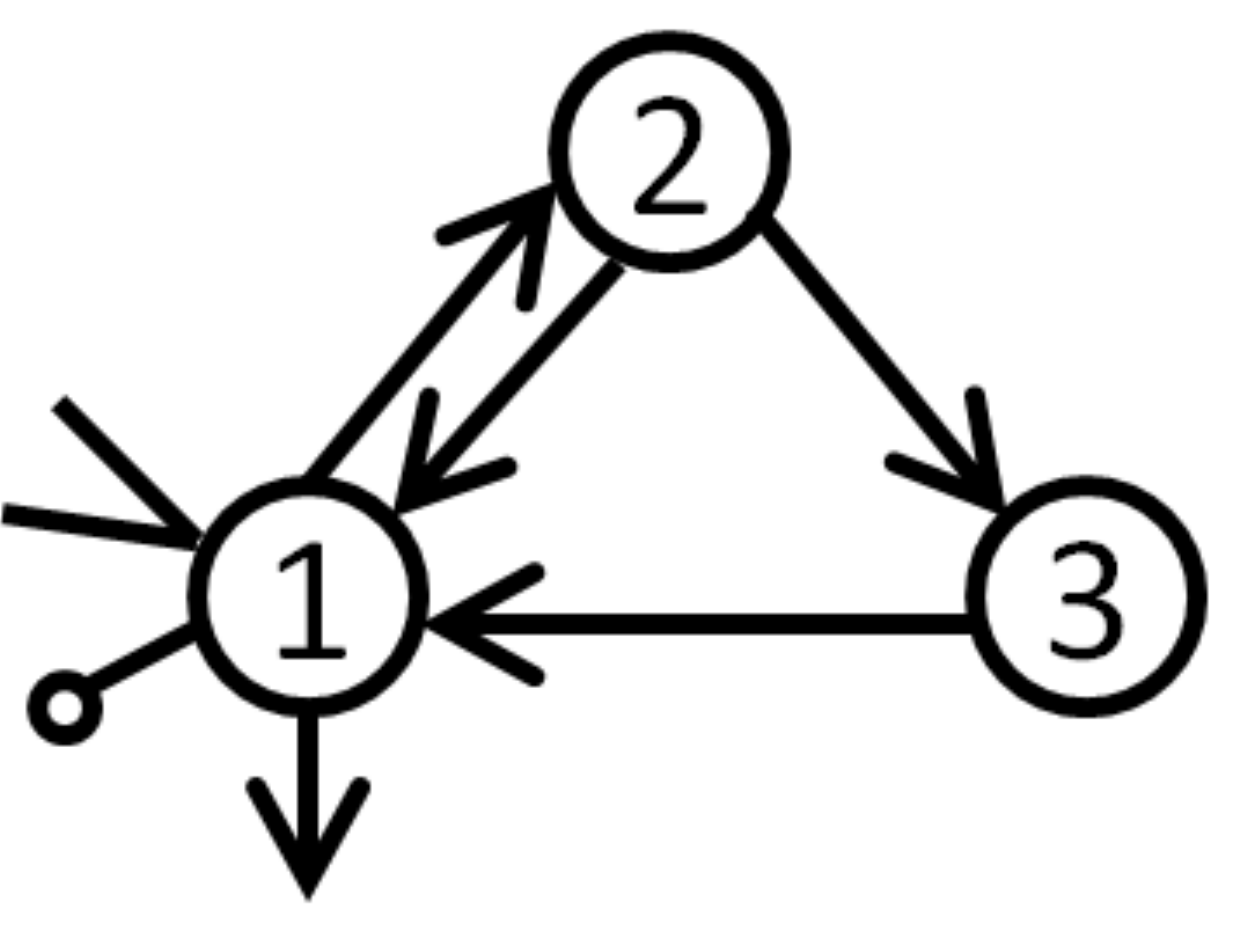}}
\end{center}\caption{A 3-compartment model}
\label{fig:3comp1leak}
\end{figure}

The compartment model drawing is displayed in Figure \ref{fig:3comp1leak}.  Note that this model can be obtained from the model in Figure \ref{fig:3comp3leaks} by removing all the leaks except the one in the first compartment.  Since the model in Figure \ref{fig:3comp3leaks} corresponds to an identifiable cycle model, then our model in Figure \ref{fig:3comp1leak} is generically locally identifiable by Theorem \ref{thm:theorem1}.  
\end{ex}


\section{Obtaining identifiability with multiple inputs and outputs}\label{sec:generalize}

In this section, we provide a complimentary result to  Theorem \ref{thm:theorem1},
that obtains identifiable models from identifiable cycle models using a combination
of removing leaks and adding inputs and outputs.
As for the proof of Theorem \ref{thm:theorem1}, the
strategy will be to gain a combinatorial understanding
of how the Jacobian of the model changes with these modifications
and use that to deduce identifiability.
The main theorem of this section is the following:

\begin{thm} \label{thm:theorem2}
Let $(G, \{1\}, \{1\}, V)$ represent an identifiable cycle model.
Let $In$, $Out$ and $Leak$ be three subsets of $V$ such
that $1 \in In$, $1 \in Out$ and $Leak \subseteq In \cup Out$.
Then $(G, In, Out, Leak)$ is generically locally identifiable. 
\end{thm}

Note that Theorem \ref{thm:theorem1} and Theorem \ref{thm:theorem2}
intersect in the special case when $In = Out = Leak  = \{1\}$.
To prove Theorem \ref{thm:theorem2}, we will need to add coefficients to our map $c$ for an identifiable cycle model, corresponding to the additional coefficients we obtain from our new input-output equations.  The first and second authors found in \cite{MeshkatSullivant} that the coefficient map $c$ factors through the cycle map $f$.  We now show that the additional coefficients of the input-output equations can also be described in terms of the graph $G$.

\begin{defn} A \textit{path} from vertex $i_{k}$ to vertex $i_{0}$ in a directed graph $G$ is a sequence of 
vertices $i_{0},i_{1}, i_{2}, \ldots, i_{k}$
such that $i_{j+1} \to i_{j}$ is an edge for all $j = 0, \ldots, k-1$.
To a path $P = i_{0},i_{1}, i_{2}, \ldots, i_{k}$, we associate the
monomial $a^{P} = a_{i_{0}i_{1}}a_{i_{1}i_{2}}\cdots a_{i_{k-1}i_{k}}$,
which we refer to as a \textit{monomial path}.  If a monomial path $a^{P}$ has
length $k$, we refer to it as a $k$-path.  
\end{defn}

Let $j$ be an input compartment and let $i$ be an output compartment, $i\neq{j}$.  Let $\mathcal{P}(i,j)$ be the set of all paths from vertex $j$ to vertex $i$.  Let $l(P)$ denote the length of a path $P \in \mathcal{P}$.  Then we can define the minimum length of all paths from vertex $j$ to vertex $i$ as $d(i,j) = min_{P \in \mathcal{P}(i,j)} l(P)$. 

\begin{prop} \label{prop:paths} Let $j$ be an input compartment and let $i$ be an output compartment, $i\neq{j}$, for a model $(G,In,Out,V)$. 
In the input-output equation corresponding to this output ($y_i$), the coefficients of the $u_j^{(k)}$ terms can be written in terms of paths from $j$ to $i$ and self-cycles.  Moreover, the coefficient of the highest order derivative term in $u_j$ can be written solely in terms of shortest paths from $j$ to $i$ as:
 $$\sum_{P\in{\mathcal{P}(i,j)}: l(P)=d(i,j)} a^P.$$ 
\end{prop}

\begin{proof} 
Recall the input-output equation corresponding to $y_i$ is:
$$
\det(\partial{I}-A)y_i=\sum_{j \in In} (-1)^{i +j}\det(A_{ji})u_j.
$$ 
where $A_{ji}$ is the submatrix of $\partial{I}-A$ obtained by deleting the $jth$ row and the $ith$ column
of $\partial{I}-A$.  Let us examine these coefficients of the $u_j^{(k)}$ terms, for $j\neq{i}$.  The characteristic polynomial of $A$ can be determined by expanding $\det(\partial{I}-A)$ along the $jth$ row.  Let $\tilde{A}$ be the matrix $A$ with the entry $a_{ji}$ nonzero.  Then for $i\neq{j}$, taking the partial derivative of the characteristic polynomial of $\tilde{A}$ with respect to $a_{ji}$ precisely gives the polynomial $\det(A_{ji})$, up to a minus sign.    Since the coefficients of the characteristic polynomial of $\tilde{A}$ factor through the cycles, then taking the derivative of these coefficients with respect to $a_{ji}$ has the effect of removing all monomial terms not involving $a_{ji}$ and setting $a_{ji}$ to one in the monomial terms that do involve $a_{ji}$.  This effectively transforms all cycles involving $a_{ji}$ to paths from the $j^{th}$ vertex to the $i^{th}$ vertex.  Thus, each of the monomial terms are products of paths from the $j^{th}$ vertex to the $i^{th}$ vertex and self-cycles.  

Let $\calc(\tilde{G})$ be the set of all cycles in $\tilde{G}$, corresponding to a matrix $\tilde{A}$.  To determine the coefficient of the highest order derivative term in $u_j$, recall that the coefficients of the characteristic polynomial of $\tilde{A}$ can be written as

$$c_{i} =  (-1)^i\sum_{C_{1}, \ldots, C_{k} \in \calc(\tilde{G})} \prod_{j = 1}^{k}  {\rm sign}(C_{j}) a^{C_{j}},$$
where the sum is over all collections of vertex disjoint cycles
involving exactly $i$ edges of $\tilde{G}$, and ${\rm sign}(C) = 1$ if
$C$ is odd length and ${\rm sign}(C) = -1$ if $C$ is even length. This means for every $i$, all cycles of length $i$ appear as monomial terms in $c_{i}$, and for $j > i$, these cycles of length $i$ appear as monomial products with other cycles in $c_{j}$.

We now determine the highest order derivative term in $u_j$.  Since $\det(A_{ji})$ is just the partial derivative of the characteristic polynomial of $\tilde{A}$ with respect to $a_{ji}$, up to a minus sign, then the right-hand side of the input-output equation for output $y_i$ is of the form:
$$
\sum_{j \in In} (-1)^{i+j}(\frac{\partial{c_{1}}}{\partial{a_{ji}}}u_{j}^{(n-1)}+\frac{\partial{c_{2}}}{\partial{a_{ji}}}u_{j}^{(n-2)}+\frac{\partial{c_{3}}}{\partial{a_{ji}}}u_{j}^{(n-3)}+\cdots+\frac{\partial{c_{n}}}{\partial{a_{ji}}}u_{j})
$$
where $|V|=n$.  Let the length of the shortest cycle involving $a_{ji}$ be of length $d(i,j)+1$, so that the length of the shortest path from $j$ to $i$ is of length $d(i,j)$.  Then the coefficient of the highest order derivative term in $u_j$ is $\partial{c_{d(i,j)+1}}/\partial{a_{ji}}$, which is a sum of the shortest paths (of length $d(i,j)$) from $j$ to $i$.  Thus it is of the form $\sum_{P\in{\mathcal{P}(i,j)}: l(P)=d(i,j)} a^P$.
\end{proof}

It is sufficient to prove Theorem \ref{thm:theorem2} for the case when $\{1\} = In \cap Out$ and $Leak \subseteq In \cup Out$, since if the model is identifiable in this case, it will remain identifiable with any additional inputs or outputs.  The main idea behind proving Theorem \ref{thm:theorem2} is the following.  For each of the inputs/outputs to compartments other than compartment $1$, we will add one additional coefficient to our map $c$, so that there are an additional $|In \cup Out| - 1$ coefficients.  Note that, in general, there are many other coefficients that arise from these additional input-output equations, but we only add one additional coefficient for each additional input/output other than compartment $1$, which is why our results will give sufficient conditions but not necessary conditions for identifiability.  We call the collection of these additional coefficients the \textit{sum of paths} map:

\begin{defn}  Let $(G,In,Out,V)$ be a linear compartment model and let $g$ be a vector corresponding to a subset of the coefficient functions of all the input-output equations derived in Theorem \ref{thm:ioeqn} having the following components:  For each $i \in Out, i \neq {1}$, $g$ contains the coefficient term in the input-output equation in $y_i$ corresponding to the sum of shortest paths from $1$ to $i$, and for each $j \in In, j \neq {1}$, $g$ contains the coefficient term in the input-output equation in $y_1$ corresponding to the sum of shortest paths from $j$ to $1$.  The function $g$ defines a map $g : \rr^{|V| + |E|} \rightarrow \rr^{|In \cup Out| - 1}$ called the \textit{sum of paths} map.  
\end{defn}

Let $\bar{c}=(c,g)$, i.e. the vector whose components are the vector $c$ of coefficient functions for an identifiable cycle model $(G, \{1\}, \{1\}, V)$ and the vector $g$ for a model $(G, In, Out, V)$ defined above.  Thus the function $\bar{c}$ defines a map $\bar{c} : \rr^{|V|+|E|} \rightarrow \rr^{2|V|+|In \cup Out|-2}$.  Similarly, let $\bar{f}=(f,g)$, i.e the vector whose components are the vector $f$ of monomial cycle functions for an identifiable cycle model $(G, \{1\}, \{1\}, V)$ and the vector $g$ for a model $(G, In, Out, V)$.   Thus the function $\bar{f}$ defines a map $\bar{f} : \rr^{|V|+|E|} \rightarrow \rr^{|E|+|In \cup Out|}$.  Then certainly we still have that $\bar{c}$ factors through $\bar{f}$, without a loss of dimension.  That is, there exists a map $\bar{\Psi} : \rr^{|E|+|In \cup Out|} \rightarrow \rr^{2|V|+|In \cup Out| - 2}$ such that we have the following commutative diagram
$$\xymatrix{
\rr^{|V|+|E|}  \ar[r]^-{\displaystyle \bar{c}} \ar@/_/[rd]^{\displaystyle \bar{f}} & 
**[r] \rr^{2|V|+|In \cup Out|-2} \\
& **[r] \rr^{|E|+|In \cup Out|} \ar[u]_{\displaystyle \bar{\Psi}}}
$$
with the property that $\bar{c}=\bar{\Psi} \circ \bar{f}$ and and $\dim {\rm image} \ \bar{c} = \dim {\rm image} \ \bar{f}$.

Passing from a model $(G, \{1\}, \{1\}, V)$ to a model $(G, In, Out, Leak)$ such
that $\{1\} = In \cap Out$ and $Leak \subseteq In \cup Out$
amounts to restricting the parameter space $\rr^{|V| + |E|}$
to a linear subspace $\Lambda \subseteq  \rr^{|V| + |E|}$ of dimension $|E| + |Leak|$
and we would like the image of $\Lambda$ under the coefficient map $\bar{c}$
to have dimension $|E| + |Leak|$.  Since  $\bar{c}$ factors through $\bar{f}$ it suffices to prove
that the image of $\Lambda$ under $\bar{f}$ has dimension $|E| + |Leak|$.

\begin{lemma} \label{lem:algindset} Let $G = (V,E)$ be a directed graph with corresponding identifiable cycle model $(G, \{1\}, \{1\}, V )$.  Consider a model $(G, In, Out, Leak)$ such
that $\{1\} = In \cap Out$ and $Leak \subseteq In \cup Out$.  Let $\bar{f} : \rr^{|V|+|E|} \rightarrow \rr^{|E|+|In \cup Out|}$ denote the map $(f,g)$ where $f$ is the cycle map and $g$ is the sum of paths map defined above. 
Let $\Lambda \subseteq \rr^{|V| + |E|}$ be the linear space satisfying
$$
\Lambda = \{ A \in \rr^{|V| + |E|}: a_{ii} =  - \sum_{j,j \neq i} a_{ji} \mbox{ for all }
i \notin{Leak} \}.
$$
Then the dimension of the image of $\Lambda$ under the map $\bar{f}$ is $|E| + |Leak|$. 
\end{lemma}

\begin{proof} 
Since $\Lambda$ is a linear space, we can replace $\bar{f}$ by the natural map  
from $\rr^{|E| + |Leak|}  \rightarrow \rr^{|E| + |In \cup Out|}$.  We also call this map $\bar{f}$.
To show that the dimension is correct, we compute the Jacobian of $\bar{f}$
and show that it has full rank.

We start with the case where $Leak=In \cup Out$.  Form the Jacobian of the map $\bar{f}$ which has columns corresponding to the $|E|+|Leak|$ parameters and rows corresponding to the $|E|-|V|+1$ cycles, the $|Leak|$ self-cycles, the $|V|-|Leak|$ non-leak diagonal entries, and the $|Leak|-1$ sums of shortest paths from input to output compartments.  Clearly, the rows corresponding to the $|Leak|$ self-cycles are linearly independent with the rest of the rows, so we focus on the $|E|$ by $|E|$ submatrix ignoring those rows and columns, and call this matrix $J$.  Arrange the matrix so that the first $|E|-|V|+1$ rows correspond to the (non-diagonal) cycles of $G$, the next $|V|-|Leak|$ rows correspond to the non-leak diagonal elements, and the last $|Leak|-1$ rows correspond to the sums of shortest paths from input to output compartments.  Let the first $|E|-|V|+1$ rows be called $A$ and the last $|V|-1$ rows be called $B$.  Clearly the rows of $A$ are linearly independent.  We want to show that the rows of $B$ are linearly independent.  Then we will show that the row space of $A$ and the row space of $B$ intersect only in the origin.  We will do this by showing that nothing in the row space of $B$ is in the kernel of the incidence matrix of the graph, $M(G)$, and thus nothing in the row space of $B$ can be described as a weighted sum of cycles (except the zero vector) from Proposition \ref{prop:kernelincidence}.  

To prove that $J$ generically has the maximal possible rank, it again suffices to find
some point where the evaluation of $J$ at this point yields
the maximal possible rank.  We choose the point where we
set all the edge parameters $a_{ij} = 1$ for all $ j \to i \in E$.
The row space of $A$ is exactly the cycle space of the graph $G$. 
The matrix $B$ is a $(|V|-1) \times |E|$ matrix.  The first $|V|-|Leak|$ rows correspond to the non-leak diagonal entries, where the (negated) row corresponding to vertex $i \in V \setminus Leak$
has a one for an edge $i' \to j'$ if and only if $i = i'$, with all other entries zero.  The last $|Leak|-1$ rows correspond to the sums of shortest paths from input vertices to output vertices, where the row corresponding the sum of all paths from $j$ to $i$ is a sum of all indicator vectors of paths from $j$ to $i$. 

We first note that in order to show that the rows of $B$ are linearly independent, i.e. there exists no $1 \times (|V|-1)$ row vector $y$ such that $yB=0$, we can instead show that there exists no $y$ such that $yBC=0$ for some matrix $C$.  In other words, we will show that the rows of this new matrix $BC$ are linearly independent.  Let $C$ be the transpose of the $|V|$ by $|E|$ incidence matrix $M(G)$ from Equation \ref{eq:eg}, where the rows of $M(G)^T$ have been arranged corresponding to columns of $B$.  In other words, the edge in the $i^{th}$ column of $B$ corresponds to the edge in the $i^{th}$ row of $M(G)^T$.  Clearly the product $BM(G)^{T}$ has $|V|-1$ rows and $|V|$ columns. 
Then
 $$BM(G)^{T}_{ij} = \sum_{k: M(G)^{T}_{kj}=1}{B_{ik}} - \sum_{k: M(G)^{T}_{kj}=-1}{B_{ik}}.$$
Thus the $(i,j)$ entry of $BM(G)^{T}$ is the sum of all entries in the $i^{th}$ row of $B$ corresponding to edges leaving vertex $j$ minus the sum of all entries in the $i^{th}$ row of $B$ corresponding to edges entering vertex $j$. 
 
Let $b$ be a row of $B$ from the first $|V|-|Leak|$ rows corresponding to the non-leak vertex $j'$.  Then
\[
  bM(G)^T_{i'} = \left\{ 
  \begin{array}{l l l}
    -1 & \quad \text{if $j' \to i' \in E$}\\
   	k & \quad \text{if $i'=j'$}\\
    0 & \quad \text{otherwise,}\\
  \end{array} \right.
\]
where $k$ is the number of edges leaving vertex $j'$.  

Now let $b$ be a row of $B$ from the last $|Leak|-1$ rows corresponding to the sum of paths from vertex $j'$ to vertex $i'$.  Then
\[
  bM(G)^T_{i} = \left\{ 
  \begin{array}{l l l}
    k & \quad \text{if $i=j'$}\\
   	-k & \quad \text{if $i=i'$}\\
    0 & \quad \text{otherwise,}\\
  \end{array} \right.
\]
where $k$ is the number of paths from vertex $j'$ to vertex $i'$.  

Let the columns of $BC$ be permuted so that the first $|V|-|Leak|$ columns correspond to the non-leak vertices ($\notin{Leak}$) and the last $|Leak|$ columns correspond to the leak vertices ($\in{Leak}$).  The first $|V|-|Leak|$ rows are thus diagonally dominant and are clearly linearly independent.  The last $|Leak|-1$ rows are linearly independent since either $i'$ or $j'$ is one and $i'\neq{j'}$.  Moreover, the $|Leak|-1$ by $|Leak|$ submatrix corresponding to the last $|Leak|-1$ rows and last $|Leak|$ columns is full rank, since the only nonzero entries in the last $|Leak|-1$ rows appear in the last $|Leak|$ columns.  We must show that the $|V|-|Leak|$ by $|V|-|Leak|$ submatrix corresponding to the first $|V|-|Leak|$ rows and columns is full rank, and thus we will have that the rows of $BC$ are linearly independent.

If the matrix is strictly diagonally dominant, then we are done.  Thus, assume the matrix is not strictly diagonally dominant.  If the matrix is not similar via a permutation to a block upper triangular matrix, then we must show that it has at least one row that is strictly diagonally dominant, and thus it will be irreducibly diagonally dominant.  If all of the rows are weakly diagonally dominant, then then means that the sum of the columns is the zero vector and thus this subgraph (corresponding to vertices in $V \setminus Leak$) is not connected to the rest of the graph, a contradiction.  Now assume the matrix is similar via a permutation to a block upper triangular matrix.  The blocks in the block upper triangular form of the transformed matrix correspond to strongly connected components of the graph.  Since the permutation similarity transformation leaves the diagonal preserved, then these blocks are each irreducibly diagonally dominant.  Thus this matrix is full rank.

Thus, the rows of $BC$ must be linearly independent, which implies the rows of $B$ must be linearly independent.  Since there exists no $y$ such that $yBC=0$ for the matrix $C=M(G)^T$, then nothing in the row space of $B$ is in the kernel of $M(G)$, and thus nothing in the row space of $B$ can be described as a weighted sum of cycles (except the zero vector) from Proposition \ref{prop:kernelincidence}.  This means the row space of $A$ and the row space of $B$ intersect only in the origin, and thus the matrix $J$ is full rank.  Thus the dimension of the image of $\Lambda$ under the map $\bar{f}$ is $|E| + |Leak|$.

For the case where $Leak \subset In \cup Out$, note that if $k$ leaks are added so that $Leak = In \cup Out$, then the matrix $J$ is full rank using the same argument as above.  Removing these $k$ leaks decreases the rank by exactly $k$, since the self-cycles $a_{ii}$ for $i \in Leak$ are not involved in any of the other components of $\bar{f}$.  This also decreases the number of parameters by exactly $k$, and thus the dimension of the image of $\Lambda$ under the map $\bar{f}$ is $|E| + |Leak|$.    
\end{proof}

\begin{cor} \label{cor:cbar} The dimension of the image of $\bar{c}$ is $|E|+|Leak|$.
\end{cor}

\begin{proof} Following the proof in Lemma \ref{lem:algindset}, we can replace the non-leaks by the self-cycles $a_{ii}$ to get that the image of the unrestricted parameter space under the map $\bar{f}$ has dimension $|E|+|Leak|$.  By the assumption that $\dim  {\rm image} \ \bar{c}$ equals $\dim  {\rm image} \ \bar{\Psi} \circ \bar{f}$, we have that the image of the unrestricted parameter space under the map $\bar{c}$ has dimension $|E|+|Leak|$.
\end{proof}

\begin{proof}[Proof of Theorem \ref{thm:theorem2}]
By Lemma \ref{lem:algindset} and the comments preceding it
we know that the image of the restricted parameter space
under the map $\bar{f}$ has dimension $|E| + |Leak|$, which 
is equal to the dimension of the image of the full parameter
space under this map.  Since, the
dimension of the image of the coefficient map $\bar{c}$ is $|E| + |Leak|$ by Corollary \ref{cor:cbar},
this must be the same for the restricted model.  But the model
has $|E| + |Leak|$ parameters, hence it is generically locally identifiable.
\end{proof}

We now demonstrate Theorem \ref{thm:theorem2} on a three compartment model with input, output, and leak all from the first compartment, as well as a leak and output from the second compartment.  

\begin{ex} 
Our ODE system has the following form:

$$
\begin{pmatrix} 
\dot{x}_1 \\
\dot{x}_2 \\
\dot{x}_3 \end{pmatrix} = {\begin{pmatrix} 
-(a_{01}+a_{21}) & a_{12} & a_{13} \\
a_{21} & -(a_{02}+a_{12}+a_{32}) & 0 \\
0 & a_{32} & -a_{13} 
\end{pmatrix}} {\begin{pmatrix}
x_1 \\
x_2 \\
x_3 \end{pmatrix} } + {\begin{pmatrix}
u_1 \\
0 \\
0 \end{pmatrix}}
$$
$$ y_1=x_1$$
$$ y_2=x_2.$$

\begin{figure}
\begin{center}
\resizebox{!}{3cm}{
\includegraphics{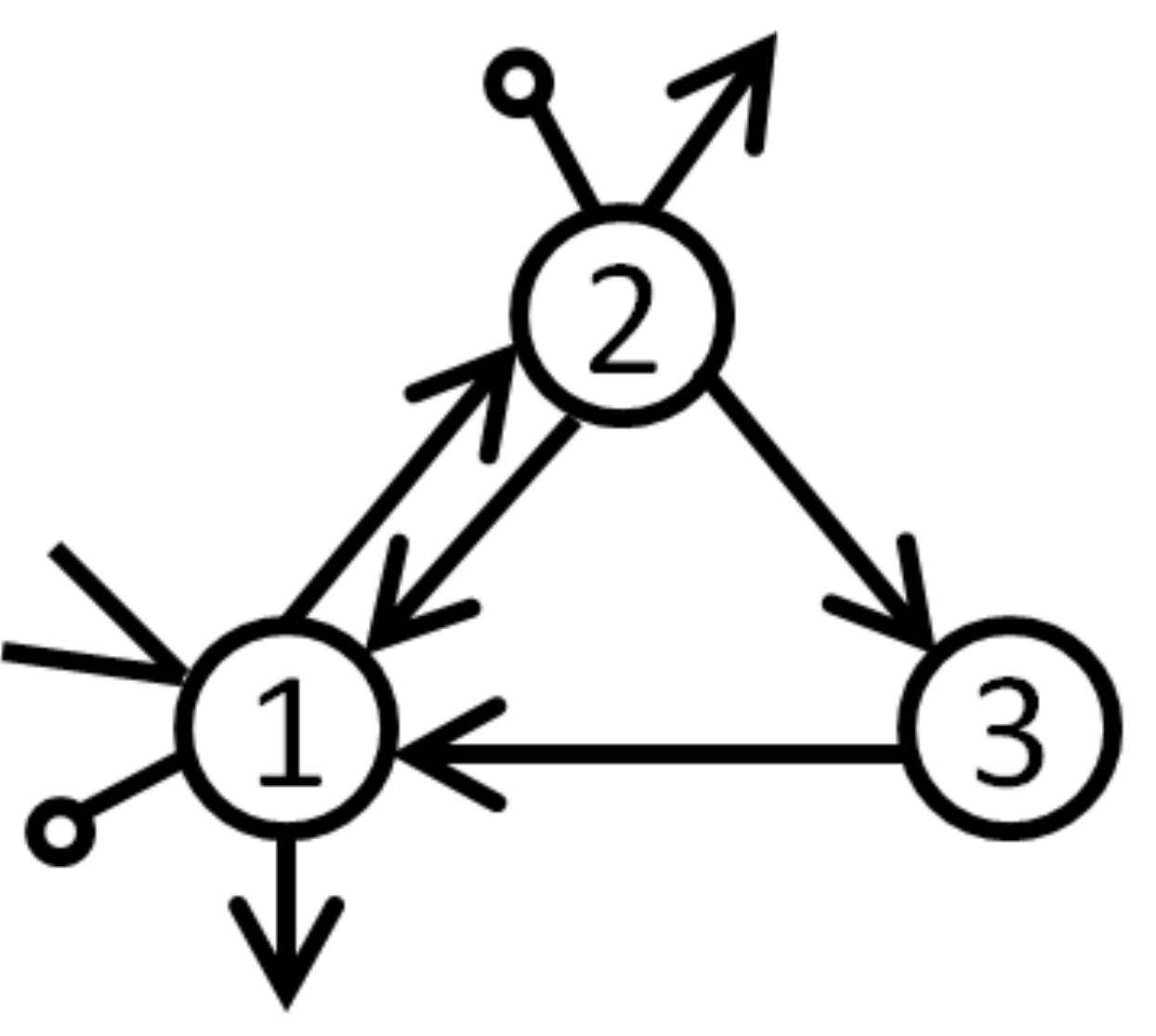}}
\end{center}\caption{A 3-compartment model}
\label{fig:3comp2leaks}
\end{figure}

The compartment model drawing is displayed in Figure \ref{fig:3comp2leaks}.  Note that this model can be obtained from the model in Figure \ref{fig:3comp3leaks} by removing the leak from the third compartment and adding an output to the second compartment, so that $In = \{1\}$, $Out=\{1,2\}$, and $Leak=\{1,2\}$.  Thus $1 \in In$, $1 \in Out$ and $Leak \subseteq In \cup Out$.  Since the model in Figure \ref{fig:3comp3leaks} corresponds to an identifiable cycle model, then our model in Figure \ref{fig:3comp1leak} is generically locally identifiable by Theorem \ref{thm:theorem2}.  
\end{ex}


\section{Identifiable submodels with One-way Flow}\label{sec:extensions}

Often, compartmental models are structured in a tiered format, wherein a series of submodels are chained together by one way flows connecting them, as shown in Figure \ref{fig:OneWayModels}. These cases are common in physiologically based pharmacokinetic models (see e.g. \cite{DiStefano1988, McMullin2003, Pilo1990}), where often one models the pharmacokinetics of a substance and its metabolites (so that each step in the metabolism of the substance forms a `tier' in the overall model). Such structures are also common in aging models, wherein individual movement or states are modeled as a single submodel, and then a discrete aging process is included, generating multiple copies of the submodel connected by a one-way flow. While these models are not strongly connected, their cascading nature makes it relatively easy to prove identifiability results. 

\begin{figure}
\centering
\includegraphics[width=0.75\textwidth]{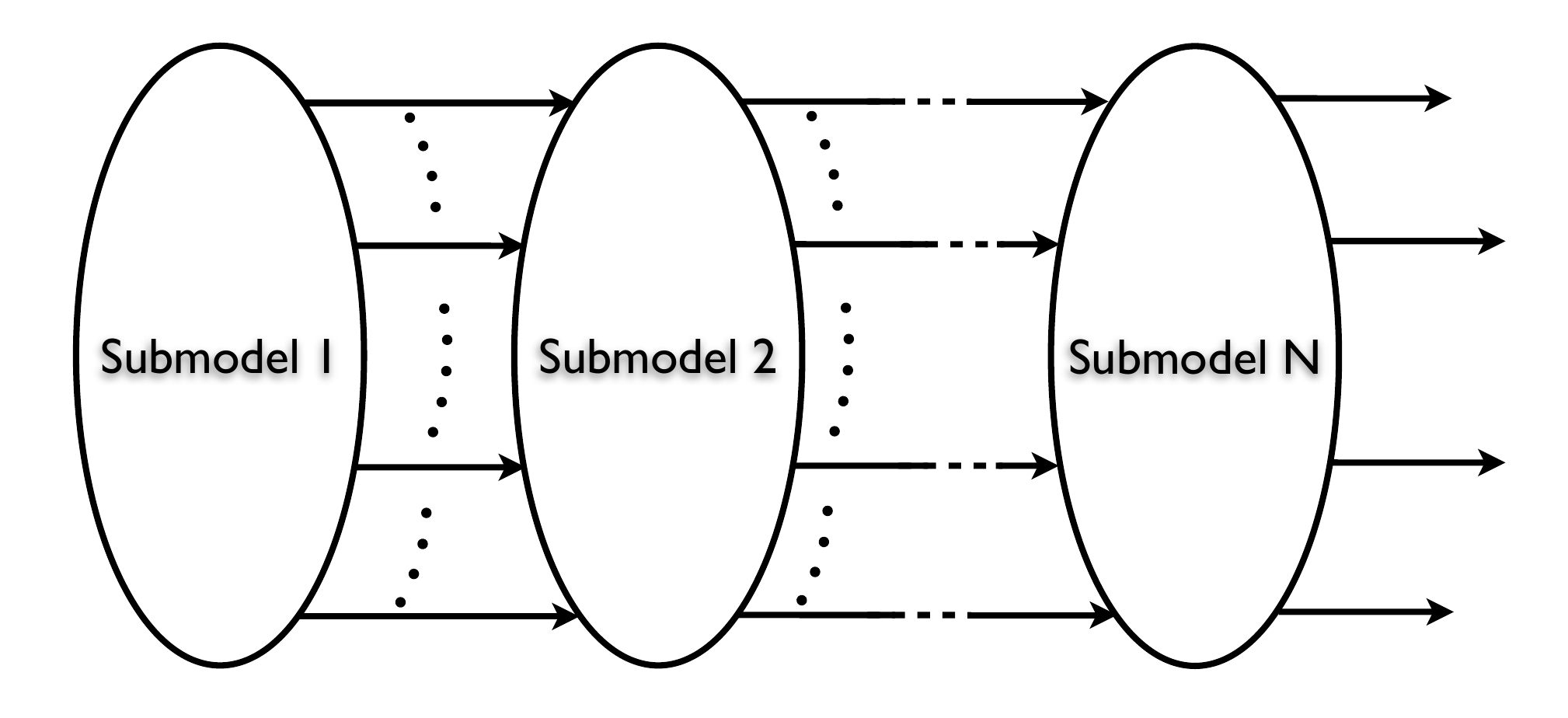}
\caption{General schematic for a one-way flow model with $N$ linear compartmental submodels. }
\label{fig:OneWayModels}
\end{figure}

We first prove a result that will allow us to determine identifiability of a tiered system by examining the identifiability of the top tiered system and then project that information to the lower tiered systems.

\begin{lemma} \label{lemma:observable} Let $G$ be a strongly connected graph corresponding to a model $(G, In, Out, Leak)$.  Then each state variable $x_j$ can be written in terms of parameters, input variables, output variables, and their derivatives.
\end{lemma}

\begin{proof}
Let our model correspond to $\dot{x}=Ax+u$ with $u_{j} \equiv 0$ for $j \notin In$ and $y_i=x_i$ for $i \in Out$.  Choose some $i \in Out$.  We will solve for $x_j$, $j \neq i$, using differential elimination.  Let $a$ be the $i^{th}$ row of $A$ with the $i^{th}$ column removed.  Let $b$ be the $i^{th}$ column of $A$ with the $i^{th}$ row removed.  Let $\tilde{A}$ be the matrix $A$ with the $i^{th}$ row and $i^{th}$ column removed.  Let $\tilde{x}$ be the vector $x$ with $x_i$ removed.  Let $\tilde{u}$ be the vector $u$ with $u_i$ removed.  Substituting $y_i=x_i$ into our system, we get that the $i^{th}$ equation is $\dot{y_i}-a_{ii}y_i-u_i=a\tilde{x}$.  Taking the derivative of this equation and then substituting in the expressions for $\dot{\tilde{x}}$, we get that $\ddot{y_i}-a_{ii}\dot{y_i}-\dot{u_i}-aby_i-a\tilde{u}=a\tilde{A}\tilde{x}$.  Taking another derivative and then substituting in the expressions for $\dot{\tilde{x}}$, we get that $\dddot{y_i}-a_{ii}\ddot{y_i}-\ddot{u_i}-ab\dot{y_i}-a\dot{\tilde{u}}-a\tilde{A}by_i-a\tilde{A}\tilde{u}=a(\tilde{A})^2\tilde{x}$.  Continuing in this way, we get that $y_i^{(m-1)}-a_{ii}y_i^{(m-2)}-u_i^{(m-2)}-(\sum^{m-3}_{j=0}{a(\tilde{A})^{m-3-j}by_i^{(j)}+a(\tilde{A})^{m-3-j}\tilde{u}^{(j)}})=a(\tilde{A})^{m-2}\tilde{x}$, where $m=|V|$.  Thus, we have a system of $m-1$ equations in $m-1$ unknowns, $B\tilde{x}=c$ where the $k^{th}$ row of $B$ is $a(\tilde{A})^{k-1}$ and the $k^{th}$ row of $c$ is $y_i^{(k)}-a_{ii}y_i^{(k-1)}-u_i^{(k-1)}-(\sum^{k-2}_{j=0}{a(\tilde{A})^{k-2-j}by_i^{(j)}+a(\tilde{A})^{k-2-j}\tilde{u}^{(j)}})$ for $k=2,...,m-1$ and just $\dot{y_i}-a_{ii}y_i-u_i$ for $k=1$.  We must show that the rows $a(\tilde{A})^{k-1}$ for $k=1,...,m-1$ are linearly independent.  Let $c_1a+c_2a\tilde{A}+...+c_{m-1}a(\tilde{A})^{m-2}=0$.  Clearly, $a$ is nonzero.  Thus, we must show that: $(1)$ $a^T$ is not an eigevector of $\tilde{A}^T$ and $(2)$ the relationship $c_1I+c_2\tilde{A}+...+c_{m-1}(\tilde{A})^{m-2}=0$ implies that $c_1=c_2=...=c_{m-1}=0$.  

We first show that $a^T$ is not an eigevector of $\tilde{A}^T$ for a generic choice of parameters.  Let $i$ be $i \in Out$ previously chosen.  If there exists a $\lambda$ such that $\tilde{A}^{T}a^T = \lambda a^T$, then we first show that $a^T$ cannot have only one nonzero entry, $a_{ij}$ for $j \neq i$.  Assume, for a contradiction, that this is the case.  This means the entries $a_{ik}$ of $a^T$ for $k \neq j$ are zero, which means there are no edges $k \rightarrow i$ for $k \neq j$.  Then the product $\tilde{A}^{T}a^T$ reduces to a product of $a_{ij}$ and the $j^{th}$ column of $A^T$ with the $i^{th}$ row removed.  Thus, the entries are of the form $a_{jk}a_{ij}$ for $k \neq i$.  But only the entry $a_{jj}a_{ij}$ should be nonzero, and thus the other entries must be zero, so that $a_{jk}$ must be zero.  This means there are no edges $k \rightarrow j$ for $k \neq i$.  Thus, there are no edges to $i$ and there are no edges to $j$ from vertices other than $i,j$.  But this means the graph $G$ is not strongly connected, a contradiction.  Now we examine the case where $a^T$ has more than one nonzero entry, $a_{ij}$ and $a_{ik}$ for $j,k \neq i$.  Let $b_j$ and $b_k$ be the $j^{th}$ and $k^{th}$ rows of $A^T$, respectively, with the $i^{th}$ column removed.  Then we have that $b_ja^T=\lambda a_{ij}$ and $b_ka^T=\lambda a_{ik}$.  Setting the expressions in $\lambda$ equal to each other and clearing denominators, we get that $a_{ik}b_ja^T=a_{ij}b_ka^T$.  We can write this as $a_{ik}(-a_{ij}a_{0j}-a_{ij}\sum_{l: j \rightarrow l \in E}{a_{lj}} + \sum_{P\in{\mathcal{P}(i,j): l(P)=2}} a^P)=a_{ij}(-a_{ik}a_{0k}-a_{ik}\sum_{l: k \rightarrow l \in E}{a_{lk}} + \sum_{P\in{\mathcal{P}(i,k): l(P)=2}} a^P)$, where $\mathcal{P}(i,j)$ and $\mathcal{P}(i,k)$ are the sets of all paths from vertex $j$ to vertex $i$ and  vertex $k$ to vertex $i$, respectively.  Then every term on either side is distinct, and no single term appears on both sides of the equation unless $j = k$.  Distributing and moving everything to one side, this means every term in the polynomial is distinct, and thus is not identically the zero polynomial.  Thus, $a^T$ is not an eigevector of $\tilde{A}^T$ except on this lower dimensional subvariety.  So for a generic choice of parameters, $a^T$ is not an eigevector of $\tilde{A}^T$.

To show that $c_1=c_2=...=c_{m-1}=0$, we must show that the minimal polynomial of $\tilde{A}$ is the characteristic polynomial of $\tilde{A}$, which is true if there are $m$ distinct eigenvalues.  Since there must be some edge from $j$ to $i$, for $i \in Out$ previously chosen and $j \neq i$, then this edge appears along the diagonal and acts like a leak, so that $\det(\tilde{A})$ is nonzero for a generic choice of parameters by Proposition \ref{prop:detnonzero}.  If the graph coresponding to $\tilde{A}$ is strongly connected, then we are done, since then the characteristic polynomial of $\tilde{A}$ is irreducible by Theorem \ref{thm:irred}, and thus it has no repeated roots.  Otherwise, the matrix $\tilde{A}$ will be a block upper triangular matrix after rearranging rows and columns of $\tilde{A}$, where each block corresponds to a strongly connected component.  Then the characteristic polynomial of $\tilde{A}$ factors as the product of the characteristic polynomials of the diagonal blocks.  Each strongly connected component is connected to the rest of the graph, and thus each block has nonzero determinant by Proposition \ref{prop:detnonzero}.  In other words, zero is not a root of any of the characteristic polynomials of the diagonal blocks.  Since each block corresponds to a strongly connected component and is connected to the rest of the graph, then each of the characteristic polynomials of the diagonal blocks are irreducible by Theorem \ref{thm:irred} and thus each characteristic polynomial has no repeated roots.  Since no two blocks have any overlapping parameters, then for a generic choice of parameters of $\tilde{A}$, there will be no common factors among any of the blocks. Thus, in either case, this means there are $m$ distinct eigenvalues.  

Thus the rows $a(\tilde{A})^{k-1}$ for $k=1,...,m-1$ are linearly independent, so there is a unique solution for $\tilde{x}$.  Thus each state variable $x_j$ can be written in terms of parameters, input variables, output variables, and their derivatives.
\end{proof}

\begin{rmk} A state variable $x_j$ in a model $M$ having the property that there exists a polynomial relationship among $x_j$, the parameters, input and output variables, and higher order derivatives of the input and output variables is called \textit{algebraically observable} \cite{Glad, Saccomani2003}. It was shown in \cite{Glad} that the order of an input-output equation is the same as the system order if and only if all variables are algebraically observable from the output.  Thus, a linear compartment model corresponding to a strongly connected graph with at least one leak has the property that all variables are observable by Corollary \ref{cor:relprime}.  In Lemma \ref{lemma:observable}, we found an explicit description of the state variables $x_j$ in terms of parameters, input variables, output variables, and their derivatives, to ensure that this relationship is linear in $x_j$.  This fact will be important for preserving the type of identifiability (i.e. global vs. local) in the tiered system, as shown in Proposition \ref{prop:tiered}.
\end{rmk}

Let $G_1=(V_1,E_1)$ and $G_2=(V_2,E_2)$ be two graphs with vertex sets $V_1=\left\{1,...,m\right\}$ and $V_2=\left\{m+1,...,n\right\}$, respectively.  Let $W_1 = \left\{i_1,...,i_k\right\}\subseteq{V_1}$ and $W_2=\left\{j_1,...,j_k\right\}\subseteq{V_2}$.  We define their union, $G_1 \cup_{W_1 \rightarrow W_2} G_2$, with vertex set $V=V_1 \cup V_2$ and edge set $E=E_1 \cup E_2 \cup E_U$ where $E_U$ is the set of all edges $i_l \to j_l$ where $i_l \in W_1$, $j_l \in W_2$, and $l=1,...,k$.  

\begin{prop} \label{prop:tiered} Let $M_1=(G_1,In_1,Out_1, Leak_1 \cup W_1)$ and $M_2=(G_2,In_2 \cup W_2,Out_2,Leak_2)$ represent two generically globally (resp. generically locally) identifiable models, where $G_1$ and $G_2$ are strongly connected, $Leak_1 \cap W_1 = \emptyset$, and $In_2 \cap W_2 = \emptyset$.  Then the corresponding union of the models, $M=(G_1 \cup_{W_1 \rightarrow W_2} G_2,In_1 \cup In_2,Out_1 \cup Out_2,Leak_1 \cup Leak_2)$ is generically globally (resp. generically locally) identifiable.  
\end{prop}

\begin{proof} 
We must show that the full model $M$ is generically identifiable, where $G_1 \cup_{W_1 \rightarrow W_2} G_2$ is defined as above.  Let $\phi_1=(x_1,...,x_m)^T$ correspond to the state variables in $M_1$ and let $\phi_2=(x_{m+1},...,x_n)^T$ correspond to the state variables in $M_2$.  Let $\dot{x}=Ax+u$ where $u_j=0$ for $j\notin{In_1 \cup In_2}$ and $y_i=x_i$ for $i \in{Out_1 \cup Out_2}$ be the differential equation system for $M$.  Partition the matrix $A$ in the following way:
$$ A = 
\begin{pmatrix}
A_1 & 0 \\ 
B & A_2
\end{pmatrix}
$$
where $A_1$ is the matrix $A$ restricted to the vertex set $V_1$ and $A_2$ is the matrix $A$ restricted to the vertex set $V_2$.  Note that the parameters in $B$ also appear in $A_1$.  Then the differential equation system $\dot{x}=Ax+u$ for $M$ can be partitioned into $\dot{\phi_1}=A_1\phi_1+\omega_1$ and $\dot{\phi_2}=A_2\phi_2 + B\phi_1 + \omega_2$, where $u=(\omega_1 , \omega_2)^T$.  Since $M_1$ is generically identifiable, then this means each of the parameters in the system  $\dot{\phi_1}=A_1\phi_1+\omega_1$ and $y_i=x_i$ for $i \in Out_1$ are generically identifiable, which means that the parameters in $A_1$ and thus in $B$ are generically identifiable in the full model $M$.  Since $G_1$ is strongly connected, we can use the system $\dot{\phi_1}=A_1\phi_1+\omega_1$ and $y_i=x_i$ for $i \in Out_1$ to get expressions for $\phi_1=(x_1,...,x_m)^T$ in terms of generically identifiable parameters, input variables, output variables, and their derivatives from Lemma \ref{lemma:observable}.

Thus in the system $\dot{\phi_2}=A_2\phi_2 + B\phi_1 + \omega_2$ and $y_i=x_i$ for $i \in Out_2$, each of the nonzero entries in $B\phi_1 + \omega_2$ correspond to known functions, with $u_j \in \omega_2$ nonzero if and only if $j \in In_2$ and the $i^{th}$ entry of $B\phi_1$ is nonzero if and only if $m+i \in W_2$.  Since $A_2$ corresponds to the graph $G_2$ with leak set $Leak_2$, then $\dot{\phi_2}=A_2\phi_2 + B\phi_1 + \omega_2$ and $y_i=x_i$ for $i \in Out_2$ defines the model $M_2=(G_2,In_2 \cup W_2,Out_2,Leak_2)$.  Thus each of the parameters in $A_2$ are generically identifiable in the full model $M$ since $M_2$ is generically identifiable.
\end{proof}

\begin{ex} \label{ex:tieredex}
We now present an example of a generically identifiable model which can be broken into two generically identfiable submodels with a one-way flow.  Our ODE system has the following form:

$$
\begin{pmatrix} 
\dot{x}_1 \\
\dot{x}_2 \\
\dot{x}_3 \\
\dot{x}_4 \\
\dot{x}_5 \end{pmatrix} = {\begin{pmatrix} 
-(a_{01}+a_{21}+a_{41}) & a_{12} & 0 & 0 & 0 \\
a_{21} & -(a_{12}+a_{32}) & a_{23} & 0 & 0 \\
0 & a_{32} & -a_{23} & 0 & 0 \\
a_{41} & 0 & 0 & -a_{54} & a_{45} \\
0 & 0 & 0 & a_{54} & -(a_{05}+a_{45})
\end{pmatrix}} {\begin{pmatrix}
x_1 \\
x_2 \\
x_3 \\
x_4 \\
x_5 \end{pmatrix} } + {\begin{pmatrix}
0 \\
u_1 \\
0 \\
0 \\
0 \end{pmatrix}}
$$
$$ y_1=x_2$$
$$ y_2=x_4.$$

\begin{figure}
\begin{center}
\resizebox{!}{3cm}{
\includegraphics{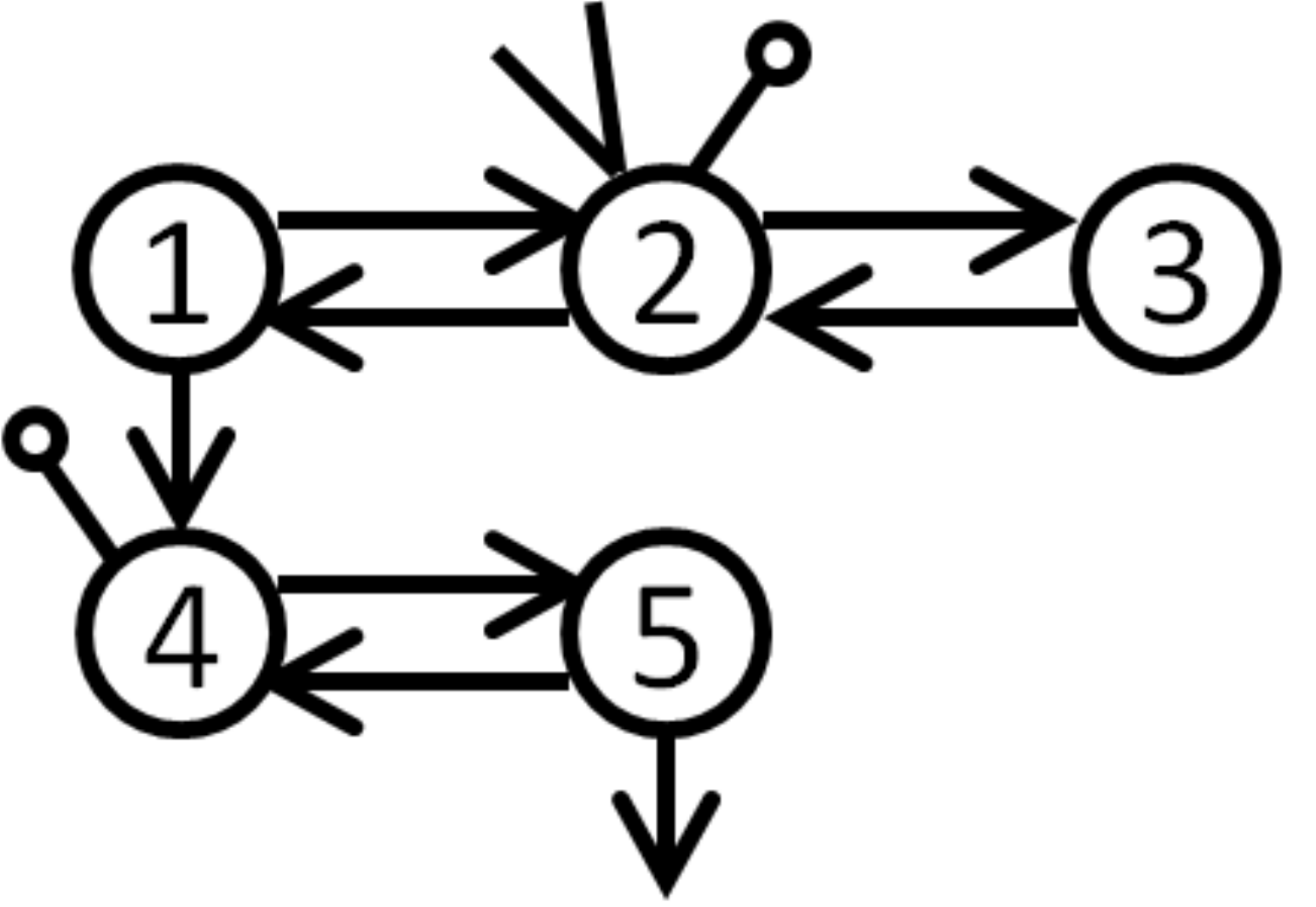}}
\end{center}\caption{A 5-compartment model}
\label{fig:5comptiered}
\end{figure}

The compartment model drawing is displayed in Figure \ref{fig:5comptiered}.  Here $M_1=(G_1,\{2\},\{2\}, \{1\})$ and $M_2=(G_2,\{4\},\{4\},\{5\})$, where $G_1$ is the induced graph on vertices $\{1,2,3\}$, $G_2$ is the induced graph on vertices $\{4,5\}$, and $Leak_1=In_2=\emptyset$.  Note that the graphs $G_1$ and $G_2$ are inductively strongly connected, thus local identifiability can be achieved for the full model using Theorem \ref{thm:theorem1}.

\end{ex}

\begin{rmk} We note that if $a_{11}=a_{33}$ in Example \ref{ex:tieredex}, then every vector is an eigenvector of $\tilde{A_1}^T$ from Proposition \ref{prop:tiered}, and thus the theorem fails because the differential elimination step does not work.  Thus, this theorem holds only for a generic choice of parameters.
\end{rmk}


\section{Examples}\label{sec:examples}

We demonstrate our results on some larger example models with applications to medicine and biology. We begin with two inductively strongly connected examples. 

\begin{ex} \label{ex:mnpk} \textit{Toxicokinetic Model of Manganese in Rats}. We consider a variation of the 11-compartment physiologically based pharmacokinetic model of manganese (Mn) absorption, distribution, and clearance in rats by Douglas et al. \cite{Douglas2010, DouglasThesis}, which was used to address questions regarding clearance of manganese and the potential links between manganism and Parkinson's disease. The network structure of the original model presented in \cite{Douglas2010, DouglasThesis} is not strongly connected, and the strongly connected component has too many edges for application of the theorems here. Thus, we modify the model slightly to include: bidirectional edges for the olfactory epithelium component (noting that diffusion would allow for at least very small reverse direction arrows), a blood-brain barrier/overall brain tissue compartment, and a simplified intestinal absorption pathway, as shown in Figure \ref{fig:MnModel}. The model equations are given by:

$$
\begin{aligned}
\left(
\begin{smallmatrix}
\dot{x}_1\\
\dot{x}_2\\
\dot{x}_3\\
\dot{x}_4\\
\dot{x}_5\\
\dot{x}_6\\
\dot{x}_7\\
\dot{x}_8\\
\dot{x}_9\\
\dot{x}_{10}\\
\dot{x}_{11}\\
\end{smallmatrix} \right)
&=
\left ( \begin{smallmatrix}
a_{11} & a_{12} & a_{13} & a_{14} & 0 & a_{16} & 0 & a_{18} & 0 & 0 & 0\\
a_{21} & a_{22} & 0          & 0           &0 & 0           & 0 &0           & 0& 0 &0\\
a_{31} & 0          & a_{33} & 0 & 0 & 0 & 0 & 0 & 0 & 0 & 0\\
a_{41} & 0 & 0 & a_{44} & 0 & 0 & 0 & 0 & 0 & 0 & 0\\
0 & 0 & 0 & 0 & a_{55} & a_{56} & 0 & 0 & 0 & 0 & 0\\
a_{61} & 0 & 0 & 0 & a_{65} & a_{66} & 0 & 0 & 0 & 0 & 0 \\
0 & 0 & 0 & 0 & 0 & 0 & a_{77} & a_{78} & 0 & 0 & 0\\
a_{81} & 0 & 0 & 0 & 0 & 0 & a_{87} & a_{88} & a_{89} & 0 & a_{811}\\
0 & 0 & 0& 0& 0&0& 0& a_{98} & a_{99} &0&0\\
0 &0&0&0&0&0&0&0&0&a_{1010} & a_{1011}\\
0 & 0&0&0&0&0&0&a_{118}&0&a_{1110}&a_{1111}
\end{smallmatrix} \right )
\left( \begin{smallmatrix}
x_1\\
x_2\\
x_3\\
x_4\\
x_5\\
x_6\\
x_7\\
x_8\\
x_9\\
x_{10}\\
x_{11}\\
\end{smallmatrix} \right)
+ \mathbf{u}
\end{aligned}
$$
where the variable numbering is as given in Figure \ref{fig:MnModel}, and each $a_{ii}$ is the negative sum of all outflows from the $i$th compartment, plus the leak from the gut in the case of $x_5$. The input vector $\mathbf{u}$ is a column vector of potential inputs, in this case inhaled manganese (inputs to $x_1$ and $x_{10}$) (Fig. \ref{fig:MnModel}). The outputs $y$ vary depending on the experiment, for example blood Mn levels considered here ($y = x_1$), or the more extensive imaging study data shown in \cite{Douglas2010, DouglasThesis} (in which case nearly all compartments are measured) both being possibilities. However, we note that as Mn can only be cleared via the gut (i.e. the model has only one leak, at $x_5$), manganese input and measurement from any compartment will result in structural identifiability, as explained below (though additional inputs or measurements will likely improve practical identifiability). If urinary excretion \cite{Greger1990, Scheuhammer1982} was to be included in the model (generating a leak from $x_4$), then measuring both urinary Mn concentrations or kidney imaging data together with fecal Mn or gut imaging data would ensure structural identifiability. 

The resulting model in Figure \ref{fig:MnModel} is inductively strongly connected with respect to vertex $1$ with exactly $2|V|-2=20$ edges, and has an input and output in the first compartment.  Thus with only one leak in the fifth compartment, the model is generically locally identifiable.  Notice there is an additional input (to the tenth compartment) that could be removed and resulting model would still be generically locally identifiable.  We note that this model is particularly large, and determining local identifiability via a symbolic algebra package, e.g. DAISY \cite{Bellu2007}, can take a very long time to compute or possibly fail.  Thus, our main result allows one to determine the local identifiability of this model without any symbolic computation at all.

\begin{figure}
\centering
\includegraphics[width=0.5\textwidth]{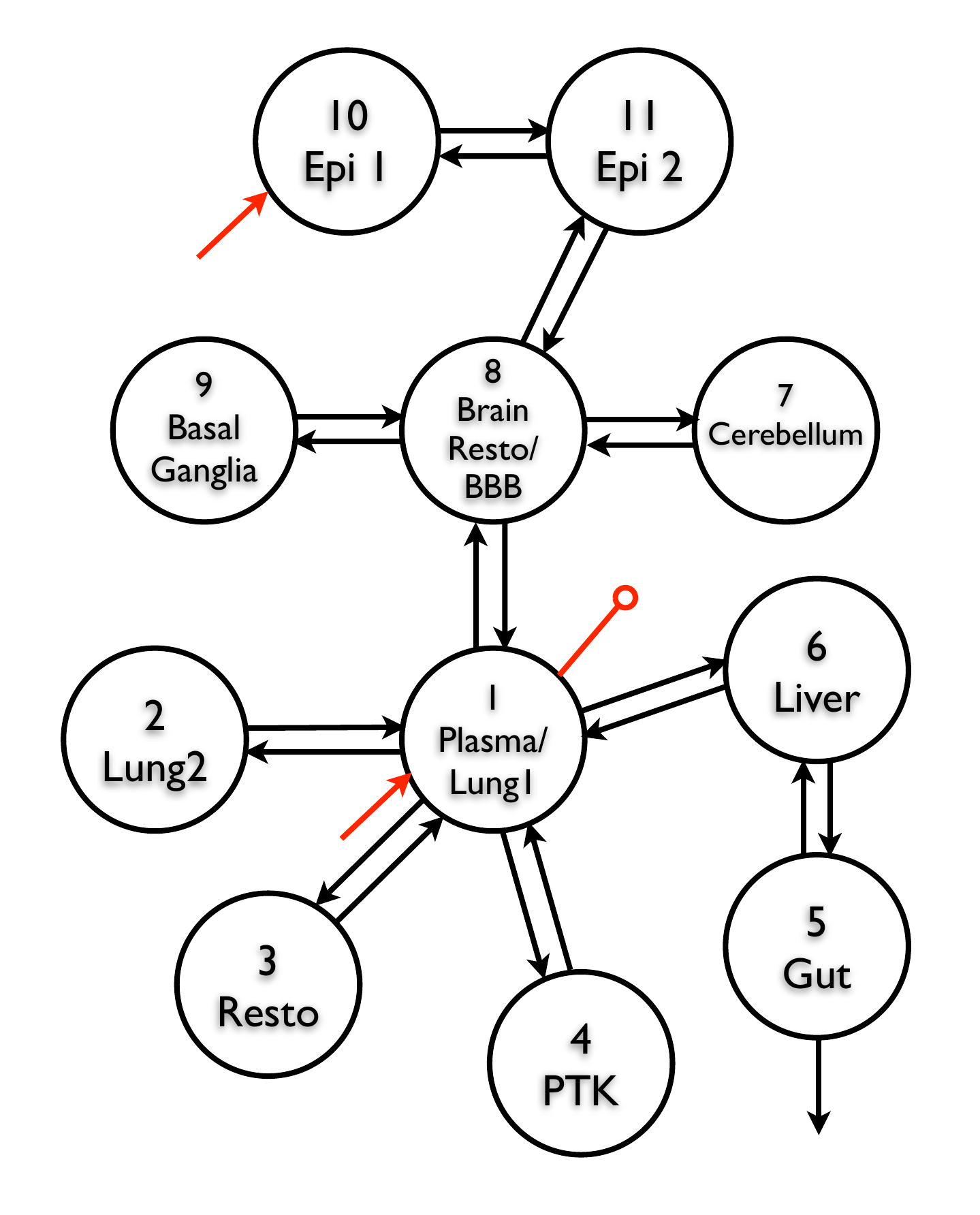}
\caption{Compartmental model of manganese pharmacokinetics in rats, based on \cite{Douglas2010,DouglasThesis}.}
\label{fig:MnModel}
\end{figure}


\end{ex}

\begin{ex} \textit{Tree models}. Compartmental networks which form trees (have no cycles) are often considered, e.g. for diffusion models along rivers and streams, or models of neuronal dendritic trees (e.g. \cite{Bressloff1993}). Assuming bidirectional flow on the trees yields an inductively strongly connected network, so that in general these models are generically identifiable when there are inputs or outputs in all compartments where leaks are present, with one leak compartment having both input and output. 
\end{ex}

Next, we demonstrate an application of our results to a model which is not inductively strongly connected, but is an identifiable cycle model.

\begin{ex} \label{ex:mnpk} \textit{Endosomal Trafficking Dynamics Model}. We consider the five-compartment model of receptor activation and trafficking between the cell membrane and endosome developed by Hori et al. \cite{Hori2006}. Hori et al. apply the model to understanding hepatic insulin receptor dynamics, though their overall model structure is broadly applicable to a range of receptor-ligand trafficking scenarios. In general, receptors and ligands could potentially be degraded in any of the five states, yielding potential leaks from all five compartments (which leaks exist will depend on the scenario considered). The resulting model is shown in Figure \ref{fig:EndosomeModel}, with equations:
$$
\begin{aligned}
\begin{pmatrix}
\dot{x}_1\\
\dot{x}_2\\
\dot{x}_3\\
\dot{x}_4\\
\dot{x}_5\\
\end{pmatrix}
&=
\begin{pmatrix}
a_{11} &      0   &  0   &a_{14}&a_{15}\\
a_{21} &a_{22}&  0    &    0    & 0\\
0          &a_{32}&a_{33}& 0 & 0\\
0 & 0 & a_{43} & a_{44} & 0\\
0 & 0 & 0 & a_{54} & a_{55}
\end{pmatrix}
\begin{pmatrix}
x_1\\
x_2\\
x_3\\
x_4\\
x_5
\end{pmatrix}
+ \mathbf{u}
\end{aligned}
$$
where $\mathbf{u}$ is the vector of inputs, typically given to $x_1$ (e.g. if ligand is introduced), and each $a_{ii}$ is the negative sum of all outflows from the $i$th compartment, plus $a_{0i}$.  The original output equations in \cite{Hori2006} are in terms of sums of compartments, which could then be solved to give the trajectories of $x_1,x_2$, and $x_3$ in terms of measured data. 
We note that this model is not inductively strongly connected (though it is strongly connected), but it can be easily shown that the dimension of the image of the double characteristic polynomial map is $|E|+1=7$ if there are leaks in every compartment and input/output to the first compartment.  Thus,  given measurements from $x_1,x_2,x_3$, we could safely include leaks from $x_1, x_2$, and $x_3$ and maintain model identifiability. 

\begin{figure}
\centering
\includegraphics[width=0.5\textwidth]{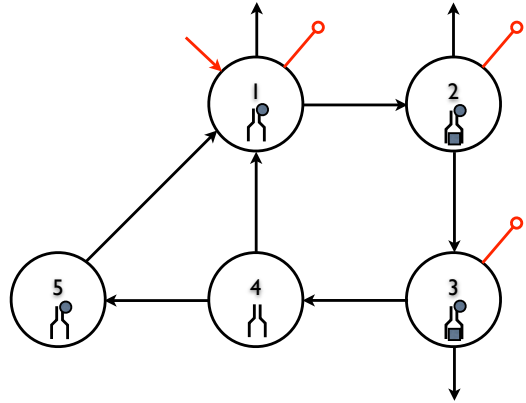}
\caption{Receptor-ligand trafficking model, based on \cite{Hori2006}. In order, the variables represent: bound surface receptor, bound phosphorylated surface receptor, internalized bound phosphorylated receptor, unbound internalized receptor, bound internalized receptor.}
\label{fig:EndosomeModel}
\end{figure}

\end{ex}


\section{Conclusions}

We have shown that, for a class of linear compartment models which we call \textit{identifiable cycle} models, changing the model to have only one leak is sufficient for local identifiability.  We have also shown that changing an identifiable cycle model to have every leak compartment correspond to an input or output compartment, with the first compartment containing both input and output, is sufficient for local identifiability.  Finally, we have shown how to combine generically identifiable models each corresponding to strongly connected graphs in a unidirectional way to obtain a generically identifiable model.

These theorems are particularly helpful for linear compartment models corresponding to inductively strongly connected graphs, for in this case no actual symbolic computation is required to test if a given model is an identifiable cycle model.  The fact that inputs can be added instead of outputs is helpful because, practically speaking, this doesn't necessarily require taking more data and it is often easier to perturb a compartment than to measure from it over time.  Even for extremely large linear models, (e.g. those occurring in hydrodynamics, machine learning applications, etc.), where it may be difficult to visually assess if a model is inductively strongly connected, symbolically testing if a model is inductively strongly connected is much faster than the Laplace transform or differential algebra methods.  However, we note that we have only proven sufficient conditions for local identifiability, and thus if our model is not inductively strongly connected or does not have input, output, and leak compartments satisfying our requirements, this does not imply the model is unidentifiable.  Future directions for this work include expanding the class of models for which a sufficient condition for identifiability exists and also finding necessary conditions for identifiable cycle models.


\section*{Acknowledgments}

Nicolette Meshkat was partially supported by the David and Lucille Packard Foundation.
Seth Sullivant was partially supported by the David and Lucille Packard 
Foundation and the US National Science Foundation (DMS 0954865).


\begin{thebibliography}{99}

\bibitem{Audoly2001} S. Audoly, G. Bellu, L. D'Angio, M. P. Saccomani, and C. Cobelli, Global identifiability of nonlinear models of biological systems, \textit{IEEE Transactions on Biomedical Engineering} {\bf 48} (1) (2001) 55-65. 

\bibitem{Anguelova} Minimal output sets for identifiability, \textit{Math. Biosci.} {\bf 239} (1) (2012) 139-153.

\bibitem{Baaijens} J. Baiijens, On the existence of identifiable reparametrizations for linear compartment models, Masters thesis, Technische Universiteit Eindhoven, 2014.

\bibitem{Bellman} R. Bellman and K. Astrom, On structural identifiability, \textit{Math. Biosci.} {\bf 7} (3-4) (1970) 329-339.

\bibitem{Bellu2007} G. Bellu, M. P. Saccomani, S. Audoly, and L. D'Angio, DAISY: A new software tool to test global identifiability of biological and physiological systems, \textit{Computer Methods and Programs in Biomedicine} {\bf 88} (1) (2007) 52-61.

\bibitem{Berman1956} M. Berman and R. Schoenfeld, Invariants in Experimental Data on Linear Kinetics and the Formulation of Models
 \textit{Journal of Applied Physics} (27) (1956) 1361-1370.
 
\bibitem{Berman1962} M. Berman, E. Shahn, M. F. Weiss, Some Formal Approaches to the Analysis of Kinetic Data in Terms of Linear Compartmental Systems, \textit{Biophysical Journal} 2 (3) (1962) 289-316.

\bibitem{Birge1969} S. J. Birge, W. A. Peck, M. Berman, and G. D. Whedon, Study of calcium absorption in man: a kinetic analysis and physiologic model, \textit{J Clin Invest.} \textbf{48} (9) (1969) 1705Ð1713.

\bibitem{Bressloff1993} P. C. Bressloff and J. G. Taylor. Compartmental-model response function for dendritic trees, \textit{Biol. Cybern.} {\bf 70} (1993), 199-207.

\bibitem{DArgenio1988} D. Z. D'Argenio, A. Schumitzky, W. Wolf, Simulation of linear compartment models with application to nuclear medicine kinetic modeling, \textit{Computer Methods and Programs in Biomedicine} \textbf{27} (1) (1988) 47-54

\bibitem{Douglas2010} P. K. Douglas, M. S. Cohen, and J. J. DiStefano III, Chronic exposure to Mn inhalation may have lasting effects: A physiologically-based toxicokinetic model in rats, \textit{Toxicology and Environmental Chemistry} {\bf 92}(2) (2010) 279-299.

\bibitem{DouglasThesis} P. K. Douglas, Physiologically based toxicokinetic modeling of manganese in rat and monkey and machine learning classification of belief vs. disbelief fMRI signals, PhD Thesis, UCLA, 2010.

\bibitem{DiStefano} J. J. DiStefano III, Dynamic Systems Biology Modeling and Simulation, Elsevier, London, 2014.

\bibitem{DiStefano1988} J. J. DiStefano III and D. Feng, Comparative aspects of the distribution, metabolism, and excretion of six iodothyronines in the rat, \textit{Endocrinology}, {\bf 123}(5) (1988) 2514-25.

\bibitem{EvansChappell} N. D. Evans and M. J. Chappell, Extensions to a procedure for generating locally identifiable reparameterisations of unidentifiable systems, \textit{Math. Biosci.} {\bf 168} (2000) 137-159.

\bibitem{Feng1991} D. Feng and J. J. Distefano III, Cut set analysis of compartmental models with applications to experiment design, \textit{American Journal of Physiology - Endocrinology and Metabolism}, \textbf{261} (2) (1991), E269-E284.

\bibitem{Glad} S. T. Glad, Differential algebraic modelling of nonlinear systems, \textit{Realization and modelling in system theory, Proceedings of the MTNS '89} {\bf 1} (1990) 97-105.

\bibitem{Greger1990} J. L. Greger, C. D. Davis, J. W. Suttie, and B. J. Lyle, Intake, serum concentrations, and urinary excretion of manganese by adult males, \textit{Am J Clin Nutr}, {\bf 51}(3) (1990) 457-461.

\bibitem{Hori2006} S. S. Hori, I. J. Kurland, J. J. DiStefano III, Role of endosomal trafficking dynamics on the regulation of hepatic insulin receptor activity: models for Fao cells, \textit{Annals of Biomedical Engineering} {\bf 34}(5) (2006) 879-892.

\bibitem{Ljung} L. Ljung and T. Glad, On global identifiability for arbitrary model parameterization, \textit{Automatica} {\bf 30}(2) (1994) 265-276.

\bibitem{McMullin2003} T. S. McMullin, J. M. Brzezicki, B. K. Cranmer, J. D. Tessari, and M. E. Andersen, Pharmacokinetic modeling of disposition and time-course studies with $[C^{14}]$atrazine, \textit{Journal of Toxicology and Environmental Health, Part A}, {\bf 66} (2003) 941-964. 

\bibitem{Meshkat2} N. Meshkat, C. Anderson, and J. J. DiStefano III, Alternative to Ritt's Pseudodivision for finding the input-output equations of multi-output models, \textit{Math. Biosci.} {\bf 239} (2012) 117-123. 

\bibitem{MeshkatSullivant} N. Meshkat and S. Sullivant, Identifiable reparametrizations of linear compartment models, \textit{Journal of Symbolic Computation} {\bf 63} (2014) 46-67.

\bibitem{Mulholland1974} R. J. Mulholland, M. S. Keener, Analysis of linear compartment models for ecosystems, Journal of Theoretical Biology, \textbf{44} (1) (1974) 105-116.

\bibitem{Pilo1990} A. Pilo, G. Iervasi, F. Vitek, M. Ferdeghini, F. Cazzuola, and R. Bianchi, Thyroidal and peripheral projection of 3,5,3'-triiodothyronine in humans by multi compartmental analysis, \textit{Am. J. Physiol.}, {\bf 258} (4, pt 1) (1990) E715-26.

\bibitem{Pohjanpalo} H. Pohjanpalo, System identifiability based on the power series expansion of the solution, \textit{Math. Biosci.} {\bf 41} (1978) 21-33.

\bibitem{Saccomani2001} M. P. Saccomani, S. Audoly, G. Bellu, and L. D'Angio, A new differential algebra algorithm to test identifiability of nonlinear systems with given initial conditions, Proceedings of the 40th IEEE Conference on Decision and Control, Orlando, Florida, USA (2001) 3108-3113.

\bibitem{Saccomani2003} M. P. Saccomani, S. Audoly, and L. D'Angio, Parameter identifiability of nonlinear systems: the role of initial conditions, \textit{Automatica} {\bf 39}(4) (2003) 619-632.


\bibitem{Scheuhammer1982} A. M. Scheuhammer and M. G. Cherian, Influence of chronic MnCl2 and EDTA treatment on tissue levels and urinary excretion of trace metals in rats, \textit{Archives of Environmental Contamination and Toxicology} {\bf 11}(2) (1982) 515-520. 

\bibitem{Stanley} R. Stanley, \textit{Enumerative Combinatorics Volume 2}, Cambridge Studies in Advanced Mathematics, \textit{62}. Cambridge University Press, 1999.

\bibitem{Tozer1981} T. N. Tozer, Concepts basic to pharmacokinetics, \textit{Pharmacology \& Therapeutics} \textbf{12} (1) (1981) 109-131. 

\bibitem{Vajda} S. Vajda, Analysis of unique structural identifiability via submodels, \textit{Math. Biosci.} {\bf 71} (2) (1984) 125-146.

\bibitem{Vicini} P. Vicini, H-T. Su, and J. J. DiStefano III, Identifiability and interval identifiability of mammillary and catenary compartmental models with some known rate constants, \textit{Math. Biosci.} {\bf 167} (2) (2000) 145-161.

\bibitem{Wagner1981} J.G. Wagner, History of pharmacokinetics, \textit{Pharmacology \& Therapeutics} \textbf{12} (3) (1981) 537-562.

\bibitem{WidmarkTandberg} E. Widmark and J. Tandberg, Uber die bedingungen f'tirdie Akkumulation Indifferenter Narkoliken
Theoretische Bereckerunger. \textit{Biochem. Z.} \textbf{147} (1924) 358-369.
 
\end{thebibliography}
\end{document}